\documentclass[12pt,letterpaper,english]{amsart}
\usepackage{lmodern}
\usepackage{helvet}

\usepackage[T1]{fontenc}
\usepackage[latin9]{inputenc}
\usepackage{mathrsfs}
\usepackage{amsthm}
\usepackage{amstext}
\usepackage{amssymb}
\usepackage{tikz}
\usepackage{esint}
\usepackage{graphicx}
\usepackage{cancel}
\usepackage{comment}
\usepackage{enumitem}
\usepackage{stmaryrd}
\usepackage{amscd,color}
\DeclareFontFamily{OT1}{pzc}{}
\DeclareFontShape{OT1}{pzc}{m}{it}{<-> s * [1.10] pzcmi7t}{}
\DeclareMathAlphabet{\mathpzc}{OT1}{pzc}{m}{it}
\usepackage[bbgreekl]{mathbbol}
\DeclareSymbolFontAlphabet{\mathbb}{AMSb}
\DeclareSymbolFontAlphabet{\mathbbl}{bbold}

\makeatletter

\pdfpageheight\paperheight
\pdfpagewidth\paperwidth

\numberwithin{equation}{section}
\numberwithin{figure}{section}
\usepackage{enumitem}		% customizable list environments
\swapnumbers
\theoremstyle{plain}
\newtheorem{thm}[equation]{Theorem}
\theoremstyle{remark}
\newtheorem{rem}[equation]{Remark}
\theoremstyle{definition}
\newtheorem{defn}[equation]{Definition}
\theoremstyle{definition}
\newtheorem{example}[equation]{Example}
\theoremstyle{plain}
\newtheorem{cor}[equation]{Corollary}
\theoremstyle{plain}
\newtheorem{prop}[equation]{Proposition}
\theoremstyle{plain}
\newtheorem{lem}[equation]{Lemma}
\theoremstyle{plain}
 \theoremstyle{definition}

\theoremstyle{definition}
\newtheorem{exs}[equation]{Examples}

\usepackage{amsfonts}
\usepackage{amssymb}
\usepackage[all]{xy}
\usepackage{array}
\usepackage{lmodern}
\usepackage[T1]{fontenc}
\usepackage{bm}

\def\d{\partial}
\def\uo{\underline{0}}
\def\ul{\underline{\ell}}

\def\uk{\underline{\kappa}}

\def\ut{\underline{t}}

\def\um{\underline{m}}

\def\ux{\underline{x}}
\def\uh{\underline{h}}
\def\ay{\mathbb{i}}
\def\QQ{\mathbb{Q}}
\def\RR{\mathbb{R}}
\def\CC{\mathbb{C}}

\def\ZZ{\mathbb{Z}}
\def\PP{\mathbb{P}}

\def\uz{\underline{z}}
\def\ut{\underline{t}}

\def\b{\bullet}

\def\V{\mathcal{V}}

\def\Y{\mathcal{Y}}
\def\X{\mathcal{X}}

\def\Y{\mathcal{Y}}
\def\A{\mathcal{A}}

\def\H{\mathcal{H}}
\def\HH{\mathbb{H}}
\def\NN{\mathbb{N}}

\def\ay{\mathbf{i}}

\def\ve{\varepsilon}

\def\IH{\mathrm{IH}}
\def\co{\mathcal{O}}

\def\gr{\mathrm{Gr}}

\def\K{\mathcal{K}}
\def\L{\mathcal{L}}
\def\E{\mathcal{E}}

\def\NP{\mathbbl{\Delta}}

\def\ua{\underline{a}}
\def\LL{\mathbb{L}}

\def\BB{\mathbb{B}}

\def\NP{\mathbbl{\Delta}}

\def\MHM{\mathrm{MHM}}
\def\sp{\mathrm{sp}}
\def\can{\mathrm{can}}
\def\van{\mathrm{van}}
\def\be{\mathbf{e}}
\def\uw{\underline{w}}
\def\sm{\setminus}
\def\ubb{\underline{\beta}}
\def\WW{\mathbb{W}}
\def\WP{\mathbb{WP}}
\def\UW{\underline{W}}
\def\UZ{\underline{Z}}
\def\ue{\underline{e}}
\def\pr{\mathrm{pr}}
\def\cs{\mathcal{S}}
\def\csb{\overline{\cs}}
\def\UX{\underline{X}}
\def\f{f'}

\theoremstyle{definition}
\newtheorem*{thx}{Acknowledgments}

\makeatother

\usepackage{babel}

\begin{document}

\title{Remarks on eigenspectra of isolated singularities}

\author[B. Castor]{Ben Castor}
\address{Department of Mathematics, Kenyon College, Gambier, OH}
\email{castor1@kenyon.edu}

\author[H. Deng]{Haohua Deng}
\address{Mathematics Department, Duke University, Durham, NC}
\email{haohua.deng@duke.edu}

\author[M. Kerr]{Matt Kerr}
\address{Washington University in St.~Louis, Department of Mathematics and Statistics, St.~Louis, MO}
\email{matkerr@wustl.edu}

\author[G. Pearlstein]{Gregory Pearlstein}
\address{Dipartimento di Matematica, Universit\`a di Pisa, Pisa, Italy}
\email{greg.pearlstein@unipi.it}

\begin{abstract}
We introduce a simple calculus, extending a variant of the Steenbrink spectrum, to describe Hodge-theoretic invariants for smoothings of isolated singularities with relative automorphisms.  After computing these ``eigenspectra'' in the quasi-homogeneous case, we give three applications to singularity bounding and monodromy of VHS.
\end{abstract}
\maketitle

%%%%%%%%%%%%%%%%%%%%%%%%%%%%%%%
%%%%%%%%%%%%%%%%%%%%%%%%%%%%%%%
\section*{Introduction}\label{S1}

Recent work of the third author and R.~Laza on the Hodge theory of degenerations \cite{KL1,KL2} re-examined the mixed Hodge theory of the Clemens-Schmid and vanishing-cycle sequences, with an emphasis on applications to limits of period maps and compactifications of moduli.  When a degenerating family of varieties has a finite group $G$ acting on its fibers, these become exact sequences in the category of mixed Hodge structures with $G\times\mu_k$-action, where $k$ is the order of $T_{\mathrm{ss}}$ (the semisimple part of monodromy).  These kinds of situations often show up in generalized Prym or cyclic-cover constructions; for instance, instead of using the period map attached to a family of varieties, one may want to use the ``exotic'' period map arising from a cyclic cover branched along the family (e.g.~\cite{ACT1,ACT2,CJL,DM,DK}).
 
In this note we explain how to encode the contributions of isolated singularities with $G$-action to the vanishing cohomology in terms of \emph{$G$-spectra} (Defn.~\ref{A11}).  These are formal sums (with positive integer coefficients) of triples in $\mathbb{Z} \times \mathbb{Q} \times \mathfrak{R}$, where $\mathfrak{R}$ is the set of irreducible representations of $G$. The term \emph{eigenspectrum} (Defn.~\ref{A12}) refers to the specific case of a cyclic group $\langle g\rangle$ with fixed generator.  (At the end of \S\ref{S3} and in most of \S\ref{S5} a larger group $\mathcal{G}$ nontrivially permutes the singularities; $G$ always denotes a subgroup stabilizing them.)

In \S\ref{S1} this formalism emerges naturally from the general setting of a proper morphism of 1-parameter degenerations over a disk, by specializing the morphism to an automorphism $g\in \mathrm{Aut}(\X/\Delta)$ fixing a singularity $x\in X_0$.  The eigenspectrum $\sigma_{f,x}^g$ simply records the dimensions of simultaneous eigenspaces of $g^*$ and $T_{\mathrm{ss}}$ in the $(p,q)$-subspaces of $V_x$ (Defn.~\ref{A12}).  We give a general computation in \S\ref{S2} of $\sigma_{f,x}^g$ in the case of a quasi-homogeneous singularity, in terms of a monomial basis for the associated Jacobian ring (Cor.~\ref{B7}). 

In the remaining sections, we give three applications.  The first, in \S\ref{S3}, is to bounding the number of nodes on Calabi-Yau hypersurfaces in weighted projective spaces (Thm.~\ref{C4}) by passing to cyclic covers. There is already a large literature on node-bounding, including \cite{JR,KL2,Mi,Sc,Va,vS2}. In the case of $\PP^{n+1}$, our approach does not improve Varchenko's bound (e.g., $135$ nodes for a quintic hypersurface in $\PP^4$), but does yield a simpler proof.  However, we do obtain the interesting result (in Thm.~\ref{C9}) that a CY hypersurface in $\PP^{n+1}$ with isolated singularities and symmetric under $\mathfrak{S}_{n+2}$ cannot contain a node whose $\mathfrak{S}_{n+2}$-orbit has cardinality $(n+2)!$ (i.e.~trivial stabilizer).

The other two applications concern codimension-one monodromy phenomena for VHSs over moduli of configurations of points and hyperplanes.  In \S\ref{S4}, the moduli space is $M_{0,2n}$, with the VHS arising from cyclic covers of $\PP^1$ branched along the $2m$ ordered points.  Propositions \ref{D5}-\ref{D6} and Example \ref{D6a} describe the eigenspectra, LMHS and monodromy types along boundary strata of certain compactifications $\overline{M}^H_{0,2n}$ due to Hassett \cite{Ha}, generalizing a computation of \cite{GKS}.  The cases $m=2,3,4,$ and $6$ go back to work of Deligne and Mostow \cite{DM} and feature a period map (isomorphism) to an arithmetic ball quotient.  While the global/extended period map is not as elegant in the remaining cases, the point is that the codimension-one boundary behavior can be dealt with uniformly and efficiently using our calculus.

Our other main example, treated in \S\ref{S5}, is the VHS $\mathcal{H}\to\mathcal{S}$ on the moduli space of general configurations of $(2n+2)$ hyperplanes in $\PP^n$, arising from the middle (intersection) cohomology of a 2:1 cover $X\to \PP^n$ branched along these hyperplanes. These are singular Calabi-Yau $n$-folds admitting a crepant resolution, and have been studied in \cite{DK,GSZ,GSVZ,SXZ}.  By passing to a smooth complete intersection $2^{2n}$-cover of $X$ and applying the Cayley trick (cf.~\cite[\S4.5]{Ke}), we replace $X$ by a smooth hypersurface $Y\subset \PP(\co_{\PP^{2n+1}}(2)^{\oplus (n+1)})$ with automorphisms by a group of order $2^{2n}$.  In codimension-one in moduli, $Y$ acquires nodes, and a variant of Schoen's result in \cite{Sc} ensures that these produce nontrivial symplectic transvections for $\H$ when $n$ is odd.  This gives an easy proof that the geometric monodromy group of $\H$ is maximal (for all $n$), and the period map ``non-classical'', a fact first proved by \cite{GSVZ} for $n=3$ and by \cite{SXZ} in general.

\subsection*{Notation}
In this paper ``MHS'' stands for $\QQ$-mixed Hodge structure.  We shall make frequent use of the \emph{Deligne bigrading} on a MHS $V$ \cite{De2}.  This is (by definition) the unique decomposition $V_{\CC}=\oplus_{p,q\in \ZZ} V^{p,q}$ with the properties that $F^k V_{\CC}=\oplus_{\substack{p,q \\ p\geq k}}V^{p,q}$, $W_{\ell}V_{\CC}=\oplus_{\substack{p,q \\ p+q\leq \ell}}V^{p,q}$, and $\overline{V^{q,p}}\equiv V^{p,q}$ mod $\oplus_{\substack{a<p \\ b<q}}V^{a,b}$.  We shall make free use of standard multi-index notation (for $n$-tuples of variables or field-elements) to simplify formulas, viz. $\uz=(z_1,\ldots,z_n)$, $\CC[\uz]=\CC[z_1,\ldots,z_n]$, $\uz^{\um}=\prod_i z_i^{m_i}$, $\um\cdot \uw=\sum_i m_iw_i$, $|\um|=\sum_i m_i$, $\ue^{(i)}=$  $i^{\text{th}}$ standard basis vector, etc.

\begin{thx}
We thank P.~Gallardo and R.~Laza for valuable discussions related to this paper, and the referee for helpful expository suggestions.  This work was partially supported by Simons Collaboration Grant 634268 and NSF Grant DMS-2101482 (MK).  MK and GP would also like to thank the Isaac Newton Institute for Mathematical Sciences for support and hospitality during the programme ``$K$-theory, algebraic cycles, and motivic homotopy theory'' (supported by EPSRC Grant Number EP/R014604/1) when work on this paper was undertaken.
\end{thx}

%%%%%%%%%%%%%%%%%%%%%%%%%%%%%%%%%%%%%%
%%%%%%%%%%%%%%%%%%%%%%%%%%%%%%%%%%%%%%
\section{$G$-spectra and eigenspectra}\label{S1}

\subsection*{\S Morphisms and mixed spectra}
We begin in the general setting of a proper morphism
\begin{equation}\label{A1}
\xymatrix{\Y \ar [rr]^{\pi} \ar [rd]_{\f} && \X \ar [ld]^f \\ & \Delta }
\end{equation}
of complex analytic spaces over a disk, which we assume is the restriction to $\Delta$ of a proper morphism of quasi-projective varieties over an algebraic curve.  (In particular, at the level of fibers we have that $\pi_t\colon Y_t\to X_t$ is a proper algebraic morphism of quasi-projective varieties.) Let $\K^{\b}\in D^b\mathrm{MHM}(\X)$ and $\L^{\b}\in D^b\mathrm{MHM}(\Y)$ be given, with a morphism $\rho\colon \K^{\b}\to R\pi_*\L^{\b}$.  Writing $\imath\colon X_0\hookrightarrow \X$ for the inclusion, the vanishing cycle triangle
\begin{equation}\label{A2}
\imath^*\overset{\mathrm{sp}}{\longrightarrow}\psi_f \overset{\mathrm{can}}{\longrightarrow}\phi_f \overset{\delta}{\underset{[+1]}{\longrightarrow}}
\end{equation}
consists of functors from $D^b\mathrm{MHM}(\X)$ to $D^b\mathrm{MHM}(X_0)$, with natural transformations between them; moreover, monodromy $T=T_{\mathrm{ss}}e^N$ induces natural automorphisms of $\psi_f$ and $\phi_f$.  By proper base-change and faithfulness of $\mathrm{rat}\colon D^b\mathrm{MHM}(X_0)\to D^b_c(X_0)$, $R\pi_*\colon D^b\MHM(Y_0)\to D^b \MHM(X_0)$ intertwines the corresponding triangle (and monodromy actions) for $(\Y,\f)$.  Taking hypercohomology on $X_0$ yields:

\begin{prop}\label{A3}
We have the commutative diagram\small
\[\xymatrix@C=13pt{\to \HH^k(X_0,\imath^*\K^{\b}) \ar [r]^{\sp}	\ar [d]^{\rho} & \HH^k(X_0,\psi_f\K^{\b}) \ar [r]^{\can} \ar [d]^{\rho} & \HH^k(X_0,\phi_f\K^{\b}) \ar [r]^{\delta\mspace{25mu}} \ar [d]^{\rho} & \HH^{k+1}(X_0,\imath^*\K^{\b})\to \ar [d]^{\rho}\\ \to \HH^k(Y_0,\imath^*\L^{\b}) \ar [r]^{\sp} & \HH^k(Y_0,\psi_{\f}\L^{\b}) \ar [r]^{\can} & \HH^k(Y_0,\phi_{\f}\L^{\b}) \ar [r]^{\delta\mspace{25mu}} & \HH^{k+1}(Y_0,\imath^*\L^{\b})\to }
\]\normalsize
with rows the vanishing-cycle \textup{(}long-exact\textup{)} sequences, in which all arrows are morphisms of MHS. Moreover, the diagram intertwines the actions of $T_{\mathrm{ss}}$ \textup{(}by automorphisms of MHS\textup{)} and $N$ \textup{(}by nilpotent $(-1,-1)$-endomorphisms of MHS\textup{)}, which are trivial \textup{(}$\mathrm{Id}$ resp.~$0$\textup{)} on the end terms.
\end{prop}

\begin{rem}\label{A4}
If $f,\f$ are themselves projective (hence proper), and $\K^{\b},\L^{\b}$ semisimple with respect to the perverse $t$-structure (e.g.~$\K^{\b}=\mathcal{IC}_{\X}^{\b}$, $\L^{\b}=\mathcal{IC}_{\Y}^{\b}$), then the Decomposition Theorem applies, yielding Clemens-Schmid sequences (cf.~\cite[\S5]{KL1}) which are then automatically compatible under $\rho$. The main consequence is that the local invariant cycle theorem holds, i.e.~$\sp$ surjects onto the $T$-invariants.
\end{rem}

Next, assume $\X,\Y,\{X_t\}_{t\neq 0}$, and $\{Y_t\}_{t\neq 0}$ are smooth, and take $\L^{\b}=\QQ_{\Y}$ and $\K^{\b}=\QQ_{\X}$; then the diagram in Prop.~\ref{A3} becomes 
\begin{equation}\label{A5}
\xymatrix{\to H^k(X_0) \ar [r]^{\sp} \ar [d]^{\pi^*} & H^k_{\lim}(X_t) \ar [r]^{\can} \ar [d]^{\pi^*} & H^k_{\van}(X_t) \ar [r]^{\delta\mspace{25mu}} \ar [d]^{\pi^*} & H^{k+1}(X_0) \to \ar [d]^{\pi^*} \\ \to H^k(Y_0) \ar [r]^{\sp}  & H^k_{\lim}(Y_t) \ar [r]^{\can}  & H^k_{\van}(Y_t) \ar [r]^{\delta\mspace{25mu}}  & H^{k+1}(Y_0) \to .}
\end{equation}
Now if $n=\dim X_0$ and $\Sigma:=\mathrm{sing}(X_0)$ is finite, then $H^k_{van}(X_t)=\{0\}$ for $k\neq n$ and, defining $V_x:=H^0\imath_x^*\phi_f\QQ_{\X}[n]$, 
\begin{equation}\label{A6}
H^n_{\van}(X_t)\cong \textstyle \bigoplus_{x\in \Sigma} V_x
\end{equation}
as MHS.  We have of course $\pi^{-1}(\Sigma)\subset \widetilde{\Sigma}:=\mathrm{sing}(Y_0)$, and if $\dim Y_0 = n$ and $|\widetilde{\Sigma}|<\infty$ then, writing $\widetilde{V}_y:=H^0\imath_y^*\phi_{\f}\QQ_{\Y}[n]$ ($y\in \widetilde{\Sigma}$), $\pi^*$ restricts to morphisms
\begin{equation}\label{A7}
[\pi^*]_x\colon V_x\to \textstyle\bigoplus_{y\in \pi^{-1}(x)}\widetilde{V}_y 
\end{equation}
of $T$-MHS --- i.e.~morphisms of MHS intertwining $T$ (hence $T_{\mathrm{ss}}$ and $N$).  These are local invariants.

Recall that $T_{\mathrm{ss}}$ acts through finite cyclic groups on each $V_x$ (and $\widetilde{V}_y$), and let $\kappa$ be the $lcm$ of their orders.  Write $\zeta_{\kappa}:=e^{\frac{2\pi\ay}{\kappa}}$ and $V_{x,\be(\frac{a}{\kappa})}^{p,q}$ for the $\be(\tfrac{a}{\kappa}):=e^{2\pi\ay\frac{a}{\kappa}}=\zeta_{\kappa}^a$-eigenspace of $T_{\mathrm{ss}}$ in $V_x^{p,q}\subset V_{x,\CC}$.  The explicit choice of $\zeta_{\kappa}\in \CC$ is needed to make the following

\begin{defn}\label{A8}
The \emph{mixed spectrum} $\sigma_{f,x}$ of the isolated singularity $x\in \Sigma$ is the element $\sum_{\alpha,w}m_{\alpha,w}^{f,x}(\alpha,w)$ of the free abelian group $\ZZ\langle \QQ\times \ZZ\rangle$, where $m_{\alpha,w}^{f,x}=\dim(V_{x,\be(\alpha)}^{\lfloor\alpha\rfloor,w-\lfloor\alpha\rfloor})$.\footnote{Here $\lfloor\cdot\rfloor$ is the greatest integer (floor) function; note also that $\be(\alpha)$ is equivalent to taking the fractional part $\{\alpha\}:=\alpha-\lfloor\alpha\rfloor$ of $\alpha$.}
\end{defn}
\noindent Evidently \eqref{A7} must be compatible with the decompositions recorded by the mixed spectra.

\subsection*{\S Automorphisms and eigenspectra}

Now let $\mathcal{G}\leq \mathrm{Aut}(\X/\Delta)$, with $\X$ and $\{X_t\}_{t\neq 0}$ smooth and $|\Sigma|<\infty$.  Applying the foregoing results with $\Y=\X$, $f=f'$, and $\pi:=g\in \mathcal{G}$, together with \cite[Prop.~5.5(i)]{KL1}, yields

\begin{cor}\label{A9}
The vanishing-cycle sequence
\begin{equation}\label{A10}
0\to H^n(X_0)	\overset{\sp}{\to}H^n_{\lim}(X_t)\overset{\can}{\to}\textstyle\bigoplus_{x\in \Sigma}V_x \overset{\delta}{\to}H^{n+1}_{\mathrm{ph}}(X_0)\to 0
\end{equation}
is an exact sequence of $\mathcal{G}\times\mu_{\kappa}$-MHS,\footnote{Again, this means that the action of $\mathcal{G}$ and $T_{\text{ss}}$ on the MHSs (as automorphisms of MHS) commute with each other and with $\sp$, $\can$, and $\delta$.} where the $\langle T_{\text{ss}}\rangle\cong \mu_{\kappa}$-action on the end terms is trivial.  If $\X/\Delta$ is proper, then $H^{n+1}_{\mathrm{ph}}(X_0):=\ker(\sp)\subseteq H^{n+1}(X_0)$ is pure of weight $n+1$.
\end{cor}

The decomposition of terms in \eqref{A10} into irreps for $\mathcal{G}\times\mu_{\kappa}$ only becomes useful if we understand the action on the vanishing cohomology $\bigoplus_{x\in\Sigma}V_x$ for a given collection of singularities.  In particular, if $gx=x$ then we need to further refine the spectrum under the resulting automorphism $g^*\colon V_x\to V_x$ of $T$-MHS.

\begin{defn}\label{A11}
Write $G\leq \mathrm{stab}(x)\leq \mathcal{G}$, and $\mathcal{R}_G$ for the set of complex irreducible representations of $G$.  The \emph{$G$-spectrum} $\sigma_{f,x}^G$ of $x$ is the element $\sum_{(\alpha,w,U)}m^{f,x,G}_{\alpha,w,U}(\alpha,w,U)$ of the free abelian group $\ZZ\langle\QQ\times \ZZ\times \mathcal{R}_G\rangle$, where (for each $(\alpha,w)$) $V_{x,\be(\alpha)}^{\lfloor\alpha\rfloor,w-\lfloor\alpha,\rfloor}\cong \bigoplus_{U\in \mathcal{R}_G}U^{\oplus m^{f,x,G}_{\alpha,w,U}}$ as $G$-representations.
\end{defn}

In the special case where $G=\langle g\rangle\cong \mu_{\ell}$ is cyclic, the $\CC$-irreps are characters indexed by the power $\zeta_{\ell}^c=e^{2\pi\ay\frac{c}{\ell}}$ of $\zeta_{\ell}$ to which $g$ is sent.

\begin{defn}\label{A12}
The \emph{eigenspectrum} of an isolated singularity $x$ with automorphism $g$ is the element $$\sigma^g_{f,x}=\sum_{(\alpha,w,\gamma)}m^{f,x,g}_{\alpha,w,\gamma}(\alpha,w,\gamma)\in \ZZ\langle \QQ\times\ZZ\times\QQ/\ZZ\rangle,$$ where $m^{f,x,g}_{\alpha,w,\gamma}$ is the dimension of the eigenspace $(V^{\lfloor\alpha\rfloor,w-\lfloor\alpha,\rfloor}_{x,\be(\alpha)})^{\be(\gamma)}\subseteq V_{x,\be(\alpha)}^{\lfloor\alpha\rfloor,w-\lfloor\alpha\rfloor}$ for $g^*$ with eigenvalue $\be(\gamma)=e^{2\pi\ay\gamma}$.
\end{defn}

\begin{rem}\label{A13}
For $\X/\Delta$ proper (with hypotheses as in Cor.~\ref{A9}), $H^n(X_t)$ is a VHS on $\Delta^*$ whose automorphism group contains $\mathcal{G}$. For any field extension $K/\QQ$, this decomposes as $K$-VHS into a direct sum of $\mathcal{G}$-isotypical components, corresponding to $K$-irreps of $\mathcal{G}$.  The $\mathcal{G}$-action on and decomposition of $H^n_{\lim}(X_t)$ obtained by taking limits are the same as those arising from the $\mathcal{G}$-MHS structure on $H^n_{\lim}(X_t)$ in Cor.~\ref{A9}.
\end{rem}

We now turn to the explicit computation of these eigenspectra in the simplest case.

%%%%%%%%%%%%%%%%%%%%%%%%%%%%%%%
%%%%%%%%%%%%%%%%%%%%%%%%%%%%%%%
\section{Quasihomogeneous singularities with automorphism}\label{S2}

Let $F\in \CC[z_1,\ldots,z_{n+1}]$ (with $n>0$) be a \emph{quasi-homogeneous polynomial} with an isolated singularity at the origin $\uo$.  That is, choosing a weight vector $\uw=(w_1,\ldots,w_{n+1})\in \QQ_{>0}^{n+1}$ and setting $$\mathfrak{M}_{\uw}:=\{\um\in \ZZ_{\geq 0}^{n+1}\mid \um\cdot \uw=1\},$$ we have 
\begin{equation}\label{B1}
F=\sum_{\um\in \mathfrak{M}_{\uw}}a_{\um}\uz^{\um}
\end{equation}
for some $a_{\um}\in \CC$.  We recall that the \emph{degree} $\kappa_F$ of $F$ is the least integer such that $\kappa_F w_i\in \NN$ for $i=1,\ldots,n+1$; define $w_i:=\kappa_F w_i$ and set $\uk:=(\kappa_1,\ldots,\kappa_{n+1})$.

Next recall the setting of Defn.~\ref{A8}, where $f\colon \X\to \Delta$ is a holomorphic map with quasi-projective fibers and smooth total space, with $X_t$ smooth for $t\neq 0$ and $\mathrm{sing}(X_0)=:\Sigma$ finite.  A singularity $x\in \Sigma\subset X_0$ is \emph{quasi-homogeneous} if $f$ can be locally analytically identified with \eqref{B1} for some $\uw$.  In that case, $V_x$ and $\sigma_{f,x}$ identify with the vanishing cohomology
\begin{equation}\label{B2}
V_F:=H^0\imath^*_{\uo}\phi_F \QQ_{\CC^{n+1}}
\end{equation}
of $F\colon \CC^{n+1}\to \CC$ at $\uo$, and its mixed spectrum $\sigma_F$.  These were first computed by Steenbrink in \cite{St}, and we briefly review the treatment from \cite[\S2]{KL2} before passing to eigenspectra.

Writing $J_F:=(\tfrac{\d F}{\d z_1},\ldots ,\tfrac{\d F}{\d z_{n+1}})\subseteq \CC[\uz]$ for the Jacobian ideal, let $\mathcal{B}\subset \ZZ_{\geq 0}^{n+1}$ be chosen so that the monomials $\{\uz^{\underline{\beta}}\}_{\underline{\beta}\in \mathcal{B}}$ provide a basis of $\CC[\uz]/J_F$.  Write $\mu_F:=|\mathcal{B}|$ for the \emph{Milnor number} of $F$, and $\ell(\underline{\beta}):=\tfrac{1}{\kappa_{F}}\sum_{i=1}^{n+1}\kappa_i(\beta_i+1)=\sum_{i=1}^{n+1}w_i(\beta_i+1)$.

\begin{prop}\label{B3}
We have $\mu_F=\dim V_F$ for the Milnor number and
$$\sigma_F=\textstyle\sum_{\underline{\beta}\in \mathcal{B}}(\alpha(\underline{\beta}),w(\underline{\beta}))\in \ZZ\langle \QQ\times \ZZ\rangle$$
for the mixed spectrum, where $\alpha(\underline{\beta}):=n+1-\ell(\underline{\beta})$ and $w(\underline{\beta}):=n$ \textup{[}resp. $n+1$\textup{]} if $\alpha(\underline{\beta})\notin \ZZ$ \textup{[}resp. $\in \ZZ$\textup{]}.
\end{prop}

\begin{proof}[Sketch]
Perform a base-change followed by weighted blow-up at $\uo$
\begin{equation}\label{B4}
\xymatrix@R=0.8pc{\CC^{n+1} \ar [d]_F & \mathfrak{X} \ar [d]_{\hat{F}} \ar [l] & \mathfrak{Y} \ar [l]_{\mathrm{Bl}_{\underline{\kappa}}} \ar [ld]^{\widetilde{F}}	\\
\CC & \Delta \ar [l] \\ t^{\kappa_F} & t \ar @{|->} [l]}
\end{equation}
with exceptional divisor $\E=\{T^{\kappa_F}=F(\underline{Z})\}\subset \mathbb{WP}[1{:}\underline{\kappa}]=:\mathbf{P}$ (in weighted homogeneous coordinates $T,Z_1,\ldots,Z_{n+1}$).  The singular fiber $\mathbb{Y}_0:=\widetilde{F}^{-1}(0)$ is the union of $\E$ and the proper transform $\widetilde{\mathbb{X}}_0$ of $\mathbb{X}_0:=F^{-1}(0)=\hat{F}^{-1}(0)$, meeting in 
$$E:=\E\cap \widetilde{\mathbb{X}}_0=\{F(\underline{Z})=0\}\subset \mathbf{H}:=\{T=0\}\,(\cong \mathbb{WP}[\underline{\kappa}])\subset \mathbf{P}.$$

The claim is then that $V_F\cong H^n(\E\setminus E)$, which can be checked using \eqref{A5} with $\pi=\mathrm{Bl}_{\underline{\kappa}}$.  Since $E$ [resp. $\uo$] is a deformation retract of $\mathbb{Y}_0$ [resp. $\mathbb{X}_0$], while $\mathbb{Y}_t=\mathbb{X}_t$ for $t\neq 0$, and $\phi_{\widetilde{F}}\QQ_{\mathcal{Y}}\simeq \imath^E_*\QQ_E(-1)[-1]$ (cf.~\cite[6.3 and 8.3-4]{KL1}), the diagram becomes
$$\xymatrix@C=1pc{&0\ar[r] & H^n_{\lim}(\mathbb{X}_t)\ar [r]^{\cong} \ar @{=} [d] & V_F \ar [d]^{\mathrm{Bl}^*} \ar [r] & 0 \\ H^{n-2}(E)(-1)\ar [r] & H^n(\E) \ar [r] & H^n_{\lim}(\mathbb{Y}_t) \ar [r] & H^{n-1}(E)(-1) \ar [r] & H^{n+1}(\E)}$$
whence the result.

Next, one constructs a basis of $H^n(\E\sm E)$ from $\mathcal{B}$, using residue theory.  Writing (with $T:=Z_0$) 
$$\Omega_{\mathbf{P}}=\sum_{j=0}^{n+1} (-1)^j Z_j\, dZ_0\wedge\cdots \wedge \widehat{dZ_j}\wedge \cdots \wedge dZ_{n+1},$$
for each $\underline{\beta}\in \mathcal{B}$ we set (with $\underline{Z}^{\underline{\beta}}=Z_1^{\beta_1}\cdots Z_{n+1}^{\beta_{n+1}}$)
\begin{equation}\label{B5}
\Omega_{\underline{\beta}}:=\frac{T^{\kappa_F}\underline{Z}^{\underline{\beta}}\Omega_{\mathbf{P}}}{T(F(\underline{Z})-T^{\kappa_F})^{\lceil \ell(\underline{\beta})\rceil}}\in \Omega^{n+1}(\mathbf{P}\sm\E\cap\mathbf{H})
\end{equation}
and $\omega_{\underline{\beta}}:=\mathrm{Res}_{\E\sm E}([\Omega_{\underline{\beta}}])\in H^n(\E\sm E)$.  We refer to \cite[Thm.~2.2]{KL2} for the proof that this has $(p,q)$-type $(\lfloor\alpha(\underline{\beta})\rfloor,\lfloor\ell(\underline{\beta})\rfloor)$, and \cite[Thm.~1]{St} for the assertion that the $\{\omega_{\underline{\beta}}\}$ give a basis.  Note that $\lfloor\alpha(\underline{\beta})\rfloor+\lfloor\ell(\underline{\beta})\rfloor=w(\underline{\beta})$.

Finally, the action of $T_{\text{ss}}$ is computed by $T\mapsto \zeta_{\kappa_F}T$, or equivalently (in weighted projective coordinates) by $Z_i\mapsto \zeta_{\kappa_F}^{-\kappa_i}Z_i=e^{-2\pi\ay w_i}Z_i$.  Clearly the effect of this on \eqref{B5} is to multiply it by $e^{2\pi\ay\sum w_i(\beta_i+1)}=e^{2\pi\ay\alpha(\underline{\beta})}$, as desired.
\end{proof}

Now given a finite group $G\leq \mathrm{Aut}(\X/\Delta)$ fixing $x\in \Sigma$, we can always choose local holomorphic coordinates on which the action is linear \cite{Ca}.  So for a given $g\in G$, we can choose coordinates to make the action diagonal, through roots of unity. Accordingly, we shall compute the eigenspectrum in the case where $g\in \mathrm{Aut}(\CC^{n+1},\uo)$ is given by
\begin{equation}\label{B6}
g(z_1,\ldots,z_{n+1}):=(\zeta_{\ell}^{c_1}z_1,\ldots,\zeta_{\ell}^{c_{n+1}}z_{n+1})	
\end{equation}
and $F\in \CC[\uz]^{\langle g\rangle}$ is a $g$-invariant quasi-homogeneous polynomial.  In fact, taking $\mathcal{B}\subset \ZZ^{n+1}_{\geq 0}$ as above, we have the

\begin{cor}\label{B7}
The eigenspectrum $\sigma_F^g$ is given by
$$\sum_{\underline{\beta}\in \mathcal{B}}(\alpha(\underline{\beta}),w(\underline{\beta}),\gamma(\underline{\beta}))\in \ZZ\langle \QQ\times\ZZ\times\QQ/\ZZ\rangle,$$	
where $\gamma(\underline{\beta}):=\frac{1}{\ell}\sum_{i=1}^{n+1}c_i(\beta_i+1)$.
\end{cor}
\begin{proof}
We only need to compute the action of $g^*$ on $\omega_{\underline{\beta}}$, which is to say the effect of $Z_i\mapsto \zeta_{\ell}^{c_i}Z_i$ on $\underline{Z}^{\underline{\beta}}\Omega_{\underline{\beta}}$.  Clearly this is just multiplication by $\zeta_{\ell}^{\sum c_i(\beta_i+1)}=e^{2\pi\ay\gamma(\underline{\beta})}$.	
\end{proof}

\begin{example}\label{B8}
For a Brieskorn-Pham singularity $F=\sum_{i=1}^{n+1}z_i^{\lambda_i}$ ($\lambda_i=\tfrac{1}{w_i}=\tfrac{\kappa_F}{\kappa_i}$), we have $\mathcal{B}=\times_{i=1}^{n+1}\{\ZZ\cap[0,d_i-2]\}$.  Hence, writing $\Gamma_m=\sum_{j=1}^{m-1}[\tfrac{j}{m}]$ in the group ring $\ZZ[\QQ]$ (with product $*$), we have $\sum_{\underline{\beta}\in \mathcal{B}}[\alpha(\underline{\beta})]=\Gamma_{\lambda_1}*\cdots*\Gamma_{\lambda_{n+1}}$. This extends to $\sum_{\underline{\beta}\in \mathcal{B}}[(\alpha(\ubb),\gamma(\ubb))]=\widetilde{\Gamma}_{\lambda_1}(\tfrac{c_1}{\ell})*\cdots*\widetilde{\Gamma}_{\lambda_{n+1}}(\tfrac{c_{n+1}}{\ell})$ in the group ring $\ZZ[\QQ\times(\QQ/\ZZ)]$ if we write $\widetilde{\Gamma}_m(\tfrac{c}{\ell})=\sum_{j=1}^{m-1}[(\tfrac{m-j}{m},\tfrac{jc}{\ell})]$.\end{example}

\begin{example}\label{B8a}
As a \emph{specific} example, consider $F=z_1^2+z_2^2+z_3^{m+1}+z_4^3$, with $g(z_1,z_2,z_3,z_4):=(z_1,z_2,z_3,\zeta_3 z_4)$. Applying \ref{B8} to compute the eigenspectrum gives $\sum_{j=1}^m[(\tfrac{5}{3}+\tfrac{j}{m+1},\tfrac{1}{3})]+\sum_{j=1}^m[(\tfrac{4}{3}+\tfrac{j}{m+1},\tfrac{2}{3})]$. 

We can interpret this scenario as a local snapshot of a 3:1 cover of $\PP^3$ branched over a cubic surface acquiring an $A_m$ singularity.  So the $\zeta_3$-eigenspace of the $(1,2)$-part of vanishing cohomology has rank equal to the number of $j$'s for which $\tfrac{5}{3}+\tfrac{j}{m+1}<2$. Since the $\zeta_3$-eigenspace of the general fiber (= cubic 3-fold) has Hodge numbers $h^{1,2}=1$ and $h^{2,1}=4$, from $\tfrac{5}{3}+\tfrac{2}{7}<2$ we see that $m$ \emph{cannot} be $\geq 6$.  This bound is sharp, since $A_5$ can occur on a cubic surface in the form $z_1^3+z_2^3-z_2z_3^2$ (see for example \cite{Sak}).  

Applying the vanishing-cycle analysis directly on a cubic surface, without passing to a triple cover and using eigenspectra, does \emph{not} rule out $A_6$.  It was this sort of phenomenon that motivated this paper.
\end{example}

\begin{rem}\label{B9}
The eigenspectrum of a $\mu$-constant (semi-quasi-\linebreak homogeneous) deformation of $(F,\gamma)$ remains constant.  Even in the more general case of \cite[\S5.2]{KL2}, one can in principle still use the action of $\gamma^*$ on the (local) Jacobian ring $\co_{n+1}/J_F$ to refine $\sigma_F$ to $\sigma_F^g$. But Corollary \ref{B7} (and quasi-homogeneous deformations of Example \ref{B8}) will suffice for our purposes below.
\end{rem}

%%%%%%%%%%%%%%%%%%%%%%%%%%%%%%%

%%%%%%%%%%%%%%%%%%%%%%%%%%%%%%%
%%%%%%%%%%%%%%%%%%%%%%%%%%%%%%%
\section{Bounding nodes on Calabi-Yau hypersurfaces}\label{S3}

It is a classical problem to bound the number of nodes (ordinary double points) on a projective hypersurface, especially for Calabi-Yau varieties.  In this section, we use eigenspectra to produce such a bound for hypersurfaces in many weighted projective spaces (Theorem \ref{C6}).  Though our emphasis is on CY varieties for illustrative purposes, it is not limited to them. In the special case of projective space, our formula recovers the bound conjectured by Arnol'd \cite{Ar} and proved by Varchenko \cite{Va} (cf.~also \cite{vS2}) by applying his semicontinuity theorem to the Bruce deformation.  This includes the famous bound of $135$ for a quintic threefold; see Example \ref{C8}.

Let $\WW=\WP[e_0:\cdots:e_{n+1}]$ be a weighted projective $(n+1)$-space with finitely many singularities.\footnote{We may assume (without loss of generality) that no $n+1$ of the $e_i$ have a common factor.} Suppose we want to bound (numbers and types of) singularities on a hypersurface $X_0=\{F_0(\UW)=0\}\subset \WW$ of degree $d$, where a smooth such hypersurface would have Hodge numbers $\uh=(h^{n,0},h^{n-1,1},\ldots ,h^{0,n})$.  Write $d_i=\tfrac{d}{e_i}$ for $i=0,\ldots,n+1$.

We shall assume that the singularities of $X_0$ are all isolated. Taking a general deformation $F_t=F_0+tG$ to produce a family of $f\colon \X\to \Delta$ with smooth total space, the vanishing-cycle sequence
\begin{equation}\label{C1}
0\to H^n(X_0)\to H^n_{\lim}(X_t)\to \textstyle\bigoplus_{x\in \Sigma}V_x \overset{\delta}{\to} H^{n+1}_{\text{ph}}(X_0)\to 0
\end{equation}
offers a naive such bound: first, by Schmid's nilpotent orbit theorem, the rank of $\gr_F^p$ remains constant in the limit, giving the second equality of 
\begin{equation}\label{C2}
h^{p,n-p}=h^{p,n-p}(X_t)=\textstyle\sum_q h^{p,q}_{\lim}(X_t)\geq \sum_q h^{p,q}(\ker(\delta)).	
\end{equation}
Moreover, the mixed spectrum $\sigma_{f,x}$ tells us the $h^{p,q}_{\zeta}(V_x)=\dim(V^{p,q}_{x,\zeta})$ (for each eigenvalue $\zeta$ of $T_{\text{ss}}$), and only the $V_{x,1}^{p,n+1-p}$ can map nontrivially under $\delta$.  Since the hyperplane class also has $T_{\text{ss}}$-eigenvalue $1$, \eqref{C2} forces $\sum_q\sum_{\zeta\neq 1}\dim(V^{p,q}_{x,\zeta})\leq h^{p,n-p}_{\text{pr}}$.

When $x$ is a node, i.e. $f\overset{\text{loc}}{\sim}\sum_{i=1}^{n+1}z_i^2$, Prop.~\ref{B3} gives $V_{x,\CC}=V^{\frac{n}{2},\frac{n}{2}}_{x,-1}$ for $n$ even and $V_{x,1}^{\frac{n+1}{2},\frac{n+1}{2}}$ for $n$ odd.  In the latter case, \eqref{C2} yields no immediate bound on the number of nodes (though one does have results like \cite[Thm.~2.9+Cor~2.11]{KL2}).  For $n=2m$ even, \eqref{C2} yields\footnote{This is by the same residue theory as used in the proof of Theorem \ref{C4} below. The notation `$*$' is from Example \ref{B8}.}
\begin{equation}\label{C3}
h^{\frac{n}{2},\frac{n}{2}}_{\text{pr}}(X_t)=\,\text{coefficient of $[\tfrac{n}{2}+1]$ in $\Gamma_{d_0}*\Gamma_{d_1}*\cdots *\Gamma_{d_{n+1}}$}
\end{equation}
as a bound, which while better than nothing is rather weak.

\begin{example}\label{CE}
The simplest nontrivial case is $\WW=\PP^3$ ($n=2$) and ($d_0=d_1=d_2=d_3=$)$\,d=4$, where $\Gamma_4^{*4}=([\tfrac{1}{4}]+[\tfrac{1}{2}]+[\tfrac{3}{4}])^{*4}=$
\begin{equation}\label{Ce}
[1]+4[\tfrac{5}{4}]+10[\tfrac{3}{2}]+16[\tfrac{7}{4}]+19[2]+16[\tfrac{9}{4}]+10[\tfrac{5}{2}]+4[\tfrac{11}{4}]+[3]
\end{equation}
correctly gives $19=h^{1,1}_{\text{pr}}(X_t)$.  This is also a poor bound for the number of nodes on a quartic surface (cf. Example \ref{C6}).
\end{example}

However, there is a simple trick which improves the bound while also giving one for odd $n$:

\begin{thm}\label{C4}
The number of nodes on $X_0$ is bounded by the coefficient, in $\Gamma_{d_0}*\Gamma_{d_1}*\cdots*\Gamma_{d_{n+1}}$, of $[\tfrac{n+1}{2}+\tfrac{1}{2d}]$ if $n$ is even and $d$ is odd, or of $[\tfrac{n+1}{2}+\tfrac{1}{d}]$ otherwise.	
\end{thm}
\begin{proof}
Let $Y_t=\{F_t(\UW)+W_{n+2}^d=0\}\subset \WP[\ue{:}1]=:\widetilde{\WW}$ be the cyclic $d{:}1$-cover of $\WW$ branched over $X_t$, with $g\colon W_{n+2}\mapsto \zeta_d W_{n+2}$ the cyclic automorphism. By Dolgachev's extension of Griffiths's residue theory \cite{Do}, a basis for the $g^*$-eigenspace $H^{n-q+1,q}_{\text{pr}}(Y_t)^{\bar{\zeta}_d^j}$ ($t\neq 0$, $0\leq j<d$) is given by the Poincar\'e residue classes $$\mathrm{Res}_{Y_t}\left(\frac{\UW^{\underline{k}-\underline{1}}W_{n+2}^{d-j-1}\Omega_{\widetilde{\WW}}}{(F_t+W_{n+2}^d)^{q+1}}\right)$$
with $k_i\in \ZZ\cap (0,d_i)$ ($i=0,\ldots,n+1$) and weights of numerator and denominator equal:  that is, $\sum_{i=0}^{n+1}e_ik_i+(d-j)=(q+1)d$, or equivalently (dividing by $d$) $\sum_{i=0}^{n+1}\tfrac{k_i}{d_i}=q+\tfrac{j}{d}$. Hence $\dim \mathrm{Gr}_F^{n-q+1} H_{\lim}^{n+1}(Y_t)^{\bar{\zeta}_d^j}=h^{n-q+1,q}(Y_t)^{\bar{\zeta}_d^j}$ is given (for $0<j<d$) by the coefficient of $[q+\tfrac{j}{d}]$ in $\Gamma_{d_0}*\cdots*\Gamma_{d_{n+1}}$.

Each node $x\in X_0$ becomes an $A_{d-1}$ singularity $y\in Y_0$, with eigenspectrum $\sum_{j=1}^{d-1}(\tfrac{n+1}{2}+\tfrac{j}{d},n+1,-\tfrac{j}{d})$ unless $n$ is even and $d$ is even (in which case the middle entry is $n+2$ at $j=\tfrac{d}{2}$). If $r$ is the number of nodes, applying \eqref{C1}-\eqref{C2} to $\Y$ and refining by $g^*$-eigenspaces therefore yields $h^{p_j,q_j}(Y_t)^{\bar{\zeta}_d^j}\geq r$ (for $0<j<d$), where $p_j=\lfloor \tfrac{n+1}{2}+\tfrac{j}{d}\rfloor$ and $q_j=n+1-p_j$.  Taking $j=1$ if $n$ is odd and $j=\lceil \tfrac{d+1}{2}\rceil$ if $n$ is even (so that $p_j=\tfrac{n+1}{2}$ resp.~$\tfrac{n}{2}+1$) yields $q_j+\tfrac{j}{d}=\tfrac{n+1}{2}+\tfrac{1}{d}$ resp.~$\tfrac{n}{2}+\tfrac{1}{d}\lceil\tfrac{d+1}{2}\rceil$, hence the claimed bound.
\end{proof}

\begin{rem}\label{C5}
As mentioned above, when $\WW=\PP^{n+1}$ this recovers Varchenko's bound \cite{Va}.  While Varchenko also uses the ``cyclic-cover trick'', our approach avoids the use of deformations and semicontinuity.
\end{rem}

\begin{example}\label{C6}
For CY hypersurfaces in $\PP^{n+1}$ ($d=n+2$), Thm.~\ref{C4} yields the bounds $3$, $16$, $135$, $1506$, and $20993$ for $n=1,2,3,4,5$, the first two of which are sharp.\footnote{The union of $3$ lines in $\PP^2$ has $3$ nodes, and a Kummer quartic $K3$ in $\PP^4$ has $16$ nodes. The bounds here are the coefficients of $[\tfrac{n+1}{2}+\tfrac{1}{n+2}]$ in $\Gamma_{n+2}^{*(n+2)}$; e.g., $16$ is the coefficient of $[\tfrac{7}{4}]$ in \eqref{Ce}.} (This is also better than what \eqref{C3} yields for $n=2$ and $4$, namely $19$ and $1751$.) It is still not known whether $135$ is sharp for quintic $3$-folds. The well-known Fermat pencil has fiber $W_0^5+\cdots+W_4^5=5W_0\cdots W_4$, with $125=|(\ZZ/5\ZZ)^3|$ nodes, while the example of van Straten \cite{vS1} with $130$ nodes remains the record.	
\end{example}

\begin{rem}\label{C7}
For $n=2$, the following bound by Miyaoka \cite{Mi} sometimes yields better results. If $X$ is any smooth projective surface which is smooth except at $r$ nodes, and $K_X$ is nef, then $r\leq 8\chi(\co_X)-\tfrac{8}{9}K_X^2$.

(a) For $X\subset \PP^3$ a surface of degree $d$, this yields the bound $\frac{4}{3}(d-1)(d-2)(d-3)+8-\tfrac{8}{9}d(d-4)^2=\tfrac{4}{9}d(d-1)^2$, which is better than Thm.~\ref{C4} for $d\geq 6$ even or $d\geq 15$ odd. A case in point is $d=6$, where \eqref{C3} gives $85$, the Theorem $68$, and Miyaoka $66$; this was further reduced to $65$ (which is sharp) by a clever use of coding theory \cite{JR}. Another is $d=8$, where we get $r\leq 174$.

(b) As a weighted projective example, one can consider surfaces $X$ of degree $10$ in $\mathbb{WP}[1{:}1{:}1{:}2]$. We have $\chi(\co_X)=1+h^2(\co_X)=35$ and $(K_X\cdot K_X)_X=(X\cdot (X+K_{\WW})^2)_{\WW}=\tfrac{10(10-5)^2}{1\cdot 1\cdot 1\cdot 2}=125$, hence $r\leq \lfloor \tfrac{1520}{9}\rfloor =168$.
\end{rem}

\begin{exs}\label{C8}
We consider some CY 3-fold hypersurfaces with $r$ nodes in weighted projective 4-folds.

(i) $X_0\subset \WP[1{:}1{:}1{:}1{:}2]$ of degree $6$: the Theorem yields $r\leq 137$, while the ``Fermat pencil'' type example $W_0^6+\cdots +W_3^6+W_4^3=3\cdot 2^{\frac{2}{3}}W_0\cdots W_4$ has $|((\ZZ/6\ZZ)^3\times \ZZ/3\ZZ)/(\ZZ/6\ZZ)|=108$ nodes.

(ii) $X_0\subset \WP[1{:}1{:}1{:}1{:}4]$ of degree $8$: the Theorem yields $r\leq 180$, while $W_0^8+\cdots +W_3^8+W_4^2=4W_0\cdots W_4$ has $|((\ZZ/8\ZZ)^3\times \ZZ/2\ZZ)/(\ZZ/8\ZZ)|=128$ nodes. Here we can improve both the bound and example, since $X_0$ is (by the quadratic formula) a double-cover of $\PP^3$ branched along an $r$-nodal octic surface. So Rem.~\ref{C7}(a) gives $r\leq 174$, while Endra\ss's example \cite{En} has $r=168$.

(iii) $X_0\subset \WP[1{:}1{:}1{:}2{:}5]$ of degree $d=10$: the Theorem yields $r\leq 169$, but because these are double covers of $WP[1{:}1{:}1{:}2]$ branched along an $r$-nodal dectic surface, Rem.~ \ref{C7}(b) reduces the bound to $168$. The standard example is $W_0^{10}+W_1^{10}+W_2^{10}+W_3^5+W_4^2=2^{\frac{4}{5}}5^{\frac{1}{2}}W_0\cdots W_4$, but this has only $100$ nodes. One can do somewhat better by taking the preimage of a Togliatti quintic \cite{Be} (with $31$ nodes avoiding the coordinate axes) under $\WP[1{:}1{:}1{:}2]\overset{1{:}2}{\twoheadrightarrow}\WP[1{:}1{:}2{:}2]\overset{1{:}2}{\twoheadrightarrow}\WP[1{:}2{:}2{:}2]\cong\PP^3$, to get $4\cdot 31=124$.
\end{exs}

In the case of a symmetric hypersurface $X_0\subset \PP^{n+1}$, cut out by $F_0\in \CC[\UW]^{\mathfrak{S}_{n+2}}$ (homogeneous of degree $d$), one can consider the family $\Y\to \Delta$ of $d$-fold cyclic covers branched along an $\mathfrak{S}_{n+2}$-invariant smoothing $\X\to \Delta$. A full accounting of this story gets into $G$-spectra ($G\cong \mu_d\times\mathrm{stab}_{\mathfrak{S}_{n+2}}(x)$) of the resulting $A_{d-1}$ singularities of $Y_0$. This leads to constraints, via character theory of $\mathfrak{S}_{n+2}$, on how $\Sigma$ can be built out of $\mathfrak{S}_{n+2}$-orbits. (However, it does not, for example, \emph{rule out} the possibility of $135$ nodes on an $\mathfrak{S}_{5}$-symmetric quintic threefold.) Here we shall only give the simplest result in this direction:

\begin{thm}\label{C9}
A symmetric CY hypersurface in $\PP^{n+1}$ \textup{(}of degree $d=n+2$\textup{)} with isolated singularities cannot contain a node with trivial stabilizer	in $\mathfrak{S}_{n+2}$.
\end{thm}
\begin{proof}
Suppose otherwise; then $Y_0$ has a set of $(n+2)!$ $A_{n+1}$ singularities with eigenspectra $\sum_{j=1}^{n+1}(\tfrac{n+1}{2}+\tfrac{j}{n+2},n+1,\tfrac{-j}{n+2})$. This contributes a subspace $V$ of dimension $(n+2)!$ to the $g^*$-eigenspace\footnote{As before, $g\colon W_{n+2}\mapsto \zeta_{n+2}W_{n+2}$ denotes the cyclic automorphism of $Y_t$.} $H^{n+1}_{\van}(Y_t)^{\zeta_{n+2}}$. It is closed under the action of $\mathfrak{S}_{n+2}$, and the triviality of the stabilizers of these $A_{n+1}$ singularities means that the trace of any $\sigma\in \mathfrak{S}_{n+2}\sm\{1\}$ is zero. So $V$ is a copy of the regular representation of $\mathfrak{S}_{n+2}$, which belongs to $\ker(\delta)\subseteq H^{\frac{n+1}{2},\frac{n+1}{2}}_{\van}(Y_t)^{\zeta_{n+2}}$. By the compatibility\footnote{This is nothing but Cor.~\ref{A9} with $\mathcal{G}=\langle g\rangle \times \mathfrak{S}_{n+2}$.} of the vanishing-cycle sequence for $\Y$ with $g^*$ and $\mathfrak{S}_{n+2}$, this forces a copy of the regular representation in $H_{\lim}^{\frac{n+1}{2},\frac{n+1}{2}}(Y_t)^{\zeta_{n+2}}$, hence $H^{\frac{n+1}{2},\frac{n+1}{2}}(Y_t)^{\zeta_{n+2}}$ for $t\neq 0$ (as $\mathfrak{S}_{n+2}$ acts on the VHS, compatibly with taking limits, cf.~Remark \ref{A13}).

Now $U:=\H^{\frac{n+1}{2},\frac{n+1}{2}}(Y_t)^{\zeta_{n+2}}$ has a basis of the form $$\eta_{\underline{k}}:=\mathrm{Res}_{Y_t}\left(\frac{\UW^{\underline{k}-\underline{1}}\Omega_{\PP^{n+2}}}{(F_0(\UW)+W^{n+2}_{n+2})^{\frac{n+3}{2}}}\right),$$
where $0<k_i<n+2$ (for $i=0,\ldots,n+1$) and (for equality of weights of numerator and denominator) $(\sum_{i=0}^{n+1}k_i)+1=\tfrac{n+3}{2}(n+2)$. Here $\mathfrak{S}_{n+2}$ acts trivially on the denominator, through the sign representation $\chi$ on $\Omega_{\PP^{n+2}}$, and by permutations of the $W_i$ on $\UW^{\underline{k}-\underline{1}}$. We claim that $U$ contains no copy of the trivial representation, \emph{a fortiori} of the regular representation, furnishing the desired contradiction.

Clearly it is equivalent to show that the representation of $\mathfrak{S}_{n+2}$ on the $\CC$-span $\widetilde{U}\,(\cong U\otimes \chi)$ of the monomials $\{\UW^{\underline{k}}\}_{\underline{k}\text{ as above}}$ contains no copy of $\chi$. Let $o:=\mathfrak{S}_{n+2}.\UW^{\underline{k}}$ be an orbit and $\widetilde{U}_{o}\subseteq \widetilde{U}$ its span. By Burnside's Lemma, $\tfrac{1}{(n+2)!}\sum_{g\in \mathfrak{S}_{n+2}}|o^g|=1$. On the other hand, $\underline{k}=(k_0,\ldots,k_{n+1})$ contains a repeated entry since there are only $n+1$ choices for each $k_i$; hence for some transposition $\tau$, $|o^{\tau}|\neq 0$. Since $\mathrm{sgn}(\tau)=-1$, this forces $\tfrac{1}{(n+2)!}\sum_{g\in \mathfrak{S}_{n+2}}\mathrm{sgn}(g)\,|o^g|$, which computes the number of copies of $\chi$ in $\widetilde{U}_o$, to be zero.
\end{proof}

For $n=1$ or $2$ this result is obvious (since $6>3$ and $24>16$), but for $n=3$, $4$, or $5$ it is less so (as $120<135$, $720<1506$, and $5040<20993$). In particular, since the examples of quintic 3-folds with $125$ and $130$ nodes are $\mathfrak{S}_5$-symmetric, and the latter has a $60$-node orbit, it is interesting that a $120$-node orbit is impossible.

%%%%%%%%%%%%%%%%%%%%%%%%%%%%%%%%%%%%
%%%%%%%%%%%%%%%%%%%%%%%%%%%%%%%%%%%%
\section{Cyclic covers of $\PP^1$}\label{S4}

In the final two sections we turn to ``codimension-one'' monodromy phenomena for period maps arising from cyclic covers. We begin with a story that generalizes elliptic curves and goes back to Deligne and Mostow \cite{DM} (see also \cite{Mn}). Given distinct points $t_1,\ldots,t_{2m}\in \PP^1$ (with projective coordinates $[S_i{:}T_i]$), define 
$$C_{\ut}:=\{[Z_0{:}Z_1{:}Z_2]\in \PP[1{:}1{:}2]\mid Z_2^m=\textstyle\prod_{i=1}^{2m}(S_iZ_1-T_iZ_0)\},$$
with automorphism $g([Z_0{:}Z_1{:}Z_2]):=[Z_0{:}Z_1{:}\zeta_m Z_2]$. For $m=2,3,4$, or $6$, the sum of $g^*$-eigenspaces $H^1(C_{\ut})^{\zeta_m}\oplus H^1(C_{\ut})^{\bar{\zeta}_m}$ produces a $\QQ$-VHS over $M_{0,2m}$,\footnote{Recall that $M_{0,n}$ parametrizes \emph{ordered} $n$-tuples of distinct points on $\PP^1$ modulo the action of $\mathrm{PSL}_2(\CC)$.} hence a period map to an arithmetic ball quotient $\Gamma\backslash\BB_{2m-3}$. This turns out to be injective,\footnote{For $m=6$ one has to quotient $M_{0,12}$ by $\mathfrak{S}_{12}$; see \cite{GKS}.} and extends to an isomorphism between GIT resp.~Hassett/KSBA compactifications of $M_{0,2m}$ and Baily-Borel resp.~toroidal compactifications of the ball quotient \cite{DM,GKS}.

So what if $m\neq 2,3,4,$ or $6$?  In the discussion that ensues, we will not be concerned with ball quotients or even the period map \emph{per se}, but only with 
\begin{itemize}[leftmargin=0.5cm]
\item the $\QQ$-VHS $\V$ over $M_{0,2m}$ arising from $H^1(C_{\ux})$, 
\item its sub-$\CC$-VHSs $\V^{\zeta_m^j}:=\ker(g^*-\zeta_m^j I)$ ($1\leq j\leq m-1$), and 
\item their limiting behavior along the boundary of the Hassett compactifications $\overline{M}_{0,[\frac{1}{m}+\epsilon]_{2m}}$ (see below). 
\end{itemize}
The point is that \emph{these can be considered uniformly for all $m\geq 2$}, not just $m=2,3,4,$ and $6$.  Moreover, using eigenspectra, \emph{we can easily compute LMHS and monodromy types along the Hassett boundary strata}, as we demonstrate in \ref{D5}-\ref{D6a}.  This is the first step toward a global study of the extended period map for this series of examples, which will necessarily go beyond the arithmetic ball quotient setting (cf.~Remark \ref{D7}).  We also refer the reader to \cite{DG}, where global partial compactifications of the period maps for some other non-Deligne-Mostow cases are constructed.

To begin with, in affine coordinates $x=\tfrac{Z_1}{Z_0}$, $y=\tfrac{Z_2}{Z_0}$, $C_{\ut}$ takes the form $y^m=f_{\ut}(x):=\prod_{i=1}^{2m}(x-t_i)$ [resp.~$\prod_{i\neq j}(x-t_i)$, if $t_j=\infty$]. While there are three possibilities for the Newton polytope $\Delta$, they all have the same interior integer points
$$(\Delta\sm\partial\Delta)\cap \ZZ^2=\{(i,j)\mid 1\leq j\leq m-1,\;1\leq i\leq 2(m-j)-1\},$$
which provide a basis of $\Omega^1(C_{\ut})$ via
$$\omega_{(i,j)}:=\mathrm{Res}_{C_{\ut}}\left(\frac{x^{i-1}y^{j-1}dx\wedge dy}{y^m-f_{\ut}(x)}\right).$$
Since $g^*\omega_{(i,j)}=\zeta_m^j\omega_{(i,j)}$, we find that 
\begin{equation}\label{D1}
\left\{\begin{array}{cc} \mathrm{rk}(\V^{\zeta_m^j})^{1,0}=2(m-j)-1, & \mathrm{rk}(\V^{\zeta_m^j})^{0,1}=2j-1 \\ \mathrm{rk}\V^{\zeta_m^j}=2m-2,\;\;\;\;\;\text{and} & \mathrm{rk}\V=2(m-1)^2. \end{array}\right.
\end{equation}
For example, if $m=5$, then $C_{\ut}$ has genus $12$; and $\V_{\CC}$ decomposes into four $\CC$-VHSs $\{\V^{\zeta_5^j}\}_{j=1}^4$ with respective Hodge numbers $(7,1)$, $(5,3)$, $(3,5)$, and $(1,7)$.

Next, we need the following:
\begin{defn}[\cite{Ha}]\label{D2}
A \emph{weighted stable rational curve} for the weight $\underline{\mu}:=(\mu_1,\ldots,\mu_n)\in \{(0,1]\cap \QQ\}^{\times n}$ is a pair\footnote{Despite the sum notation, the order of points with equal weights is retained.} $(\mathcal{C},\sum\mu_i p_i)$ with:
\begin{itemize}[leftmargin=0.5cm]
\item $\mathcal{C}$ a nodal connected projective curve of arithmetic genus $0$;
\item each $p_i$ a smooth point of $\mathcal{C}$;
\item if $p_{i_1}=\cdots =p_{i_r}$, then $\mu_{i_1}+\cdots+\mu_{i_r}\leq 1$; and
\item the $\QQ$-divisor $K_{\mathcal{C}}+\sum_{i=1}^n \mu_i p_i$ is ample (i.e.~on each irreducible component, the sum of weights plus number of nodes is $>2$).
\end{itemize}
We will write $(\mu,\ldots,\mu)=:[\mu]_n$ for repeated weights.
\end{defn}

\begin{thm}[\cite{Ha}]\label{D3}
\textup{(i)} There exists a smooth projective fine moduli space $\overline{M}_{0,\underline{\mu}}$ parametrizing $\underline{\mu}$-weighted stable rational curves, and containing $M_{0,n}$ as a Zariski-open subset.

\textup{(ii)}	 Given weights $\underline{\mu}=(\mu_1,\ldots,\mu_n)$ and $\tilde{\underline{\mu}}=(\tilde{\mu}_1,\ldots,\tilde{\mu}_n)$ with $\mu_i\leq \tilde{\mu}_i$ \textup{(}$\forall i$\textup{)}, there exists a birational \emph{reduction morphism} $\pi_{\tilde{\underline{\mu}},\underline{\mu}}\colon \overline{M}_{0,\tilde{\underline{\mu}}}\twoheadrightarrow \overline{M}_{0,\underline{\mu}}$ contracting all components which violate the ampleness property in \eqref{D2} for the weight $\tilde{\underline{\mu}}$.
\end{thm}

\begin{rem}\label{D4}
(a) $\overline{M}_{0,[1]_n}$ reproduces the Deligne-Mumford-Knudsen compactification $\overline{M}_{0,n}$.

(b)	 Although the ampleness property forces $\sum\mu_i>2$, if for $|\underline{\mu}|=2$ we \emph{define} $\overline{M}_{0,\underline{\mu}}$ to be the GIT quotient $(\PP^1)^n\sslash_{\underline{\mu}}\mathrm{SL}_2$, then \eqref{D3}(ii) extends to this case; and if we take $\tilde{\mu}_i=\mu_i+\epsilon$ ($\epsilon\in \QQ$, $0<\epsilon\ll 1$) then $\pi_{\tilde{\underline{\mu}},\underline{\mu}}$ is Kirwan's partial desingularization which blows up the strictly semistable locus.
\end{rem}

Our interest henceforth is in the equal-weight Hassett compactification $\overline{M}^H_{0,2m}:=\overline{M}_{0,[\frac{1}{m}+\epsilon]_{2m}}$ and its morphism $\pi$ to $\overline{M}^{\text{GIT}}_{0,2m}:=\overline{M}_{0,[\frac{1}{m}]_{2m}}$. As the reader may easily check, the irreducible components of $\overline{M}^H_{0,2m}\sm M_{0,2m}$ are of two types, parametrizing\footnote{More precisely, it is a dense open subset of each component that parametrizes the displayed objects.} stable weighted curves as shown (up to reordering of the $\{p_i\}$):

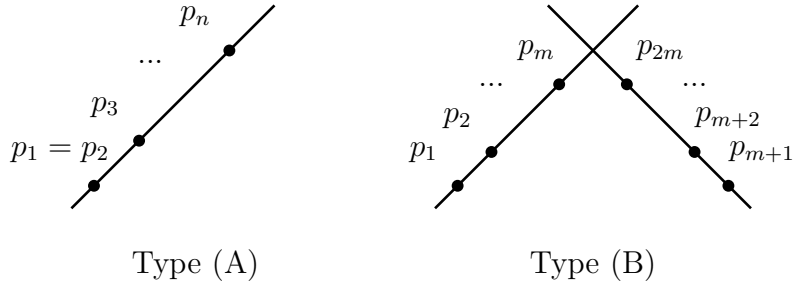
\begin{center}
\begin{equation*}
\begin{tikzpicture}[scale=3]
\draw[-,line width=1.0pt] (-0.5,-0.5) -- (0.4,0.4); 
\fill (-0.4,-0.4) circle (0.75pt);
\node at (-0.55,-0.25) {$p_1=p_2$};
\fill (-0.2,-0.2) circle (0.75pt);
\node at (-0.35,-0.05) {$p_3$};
\node at (-0.15,0.15) {$...$};
\node at (0.05,0.35) {$p_n$};
\fill (0.2,0.2) circle (0.75pt);
\node at (0.05, -0.75) {Type (A)};
\end{tikzpicture}
\mspace{70mu}
\begin{tikzpicture}[scale=3]
\draw[-,line width=1.0pt] (-0.5,-0.5) -- (0.4,0.4);
\draw[-,line width=1.0pt] (0.9,-0.5) -- (0,0.4);
\fill (-0.4,-0.4) circle (0.75pt);
\node at (-0.55,-0.25) {$p_1$};
\fill (-0.25,-0.25) circle (0.75pt);
\node at (-0.4,-0.1) {$p_2$};
\node at (-0.25,0.05) {$...$};
\node at (-0.05,0.2) {$p_m$};
\fill (0.05,0.05) circle (0.75pt);
\fill (0.8,-0.4) circle (0.75pt);
\node at (0.95,-0.25) {$p_{m+1}$};
\fill (0.65,-0.25) circle (0.75pt);
\node at (0.8,-0.1) {$p_{m+2}$};
\node at (0.65,0.05) {$...$};
\node at (0.5,0.2) {$p_{2m}$};
\fill (0.35,0.05) circle (0.75pt);
\node at (0.2,-0.75) {Type (B)};
\end{tikzpicture}
\end{equation*}
\end{center}

\noindent It is also clear that $\pi$ preserves the type (A) strata whilst contracting the type (B) ones to a (strictly semistable) point parametrizing the object
\begin{center}
\begin{equation*}
\begin{tikzpicture}[scale=3]
\draw[-,line width=1.0pt] (-0.5,0) -- (0.9,0);
\fill (0.6,0) circle (0.75pt);
\fill (-0.2,0) circle (0.75pt);
\node at (-0.3,0.15) {$p_1=...=p_m$};
\node at (0.85,0.15) {$p_{m+1}=...=p_{2m}$};
\end{tikzpicture}
\end{equation*}
\end{center}

\noindent The $\CC$-VHSs $\V^{\zeta_m^j}$ admit canonical extensions across the smooth part of $\overline{M}^H_{0,2m}\sm M_{0,2m}$, and we shall now compute the LMHS types there.

\begin{prop}\label{D5}
Along type \textup{(}A\textup{)} strata:
\begin{itemize}[leftmargin=0.5cm]
\item $\V_{\lim}^{\zeta_m^j}$ is pure of weight $1$, with $h^{1,0}=2m-2j-1$ and $h^{0,1}=2j-1$, unless $j=\tfrac{m}{2}$;
\item if $j=\tfrac{m}{2}$, then $h^{1,1}=h^{0,0}=1$, $h^{1,0}=h^{0,1}=m-1$, and $T=e^N$ \textup{(}with $N$ an isomorphism from the $(1,1)$ to $(0,0$ part\textup{)}; and
\item if $j>\tfrac{m}{2}$ \textup{[}resp.~$<\tfrac{m}{2}$\textup{]}, then we have the decomposition $\V^{\zeta_m^j}_{\lim}=\V_{\lim,1}^{\zeta_m^j}\oplus \V^{\zeta_m^j}_{\lim,\bar{\zeta}_m^{2j}}$ into $T=T_{\text{ss}}$-eigenspaces, where $\V^{\zeta_m^j}_{\lim,\bar{\zeta}_m^{2j}}$ is $1$-dimensional of type $(0,1)$ \textup{[}resp.~$(1,0)$\textup{]}.
\end{itemize}
\end{prop}
\begin{proof}
Begin by locally modeling (the effect on $C_{\ut}$ of) the collision of two points by $y^m+z^2=s$, as $s\to 0$. This has eigenspectrum
$$\textstyle\sum_{j=1}^{m-1}(\tfrac{3}{2}-\tfrac{j}{m},w(j),\tfrac{j}{m}),$$
where $w(j)=2$ if $j=\tfrac{m}{2}$ and $1$ otherwise.  Next, we apply the vanishing-cycle sequence (with $H^2_{\text{ph}}=\{0\}$ since the degenerate curve remains irreducible) to compute the LMHS. Finally, we perform a base-change by $s\mapsto s^2$ to preserve ordering of points, which squares the eigenvalues of the $T_{\text{ss}}$-action; in other words, we replace $\tfrac{3}{2}-\tfrac{j}{m}$ by $\{2(\tfrac{3}{2}-\tfrac{j}{m})\}+\lfloor \tfrac{3}{2}-\tfrac{j}{m}\rfloor$ ($\{\cdot\}$ denoting the fractional part), which gives the result.
\end{proof}

\begin{prop}\label{D6}
Along the type \textup{(}B\textup{)} strata, for each $1\leq j\leq m-1$, $\V^{\zeta_m^j}_{\lim}$ has Hodge numbers $h^{1,1}=h^{0,0}=1$, $h^{1,0}=2m-2j-2$, and $h^{0,1}=2j-2$; $N$ is an isomorphism from the $(1,1)$ to $(0,0)$ part, and $T=e^N$ is unipotent.
\end{prop}
\begin{proof}
In the GIT compactification for \emph{unordered} points, the degeneration is locally modeled by two copies of $y^m+x^m=s$, each with eigenspectrum
$$\textstyle\sum_{j=1}^{m-1}(1,2,\tfrac{j}{m})+\sum_{j=2}^{m-1}\sum_{k=1}^{j-1}(\tfrac{k+m-j}{m},1,\tfrac{j}{m})+\sum_{j=1}^{m-2}\sum_{k=j+1}^{m-1}(\tfrac{k+m-j}{m},1,\tfrac{j}{m}).$$
At this point one applies the vanishing-cycle sequence to deduce the form of the LMHS, noting that the degenerate curve is a union of $m$ $\PP^1$'s and $H^2_{\text{ph}}\cong \QQ(-1)^{\oplus m-1}$. For $\overline{M}^H_{0,2m}$, one then applies the base-change by $s\mapsto s^m$, which trivializes $T_{\text{ss}}$, allowing the extension class to vary along the type (B) stratum.
\end{proof}

\begin{example}\label{D6a}
Combining \eqref{D1} with the two Propositions, $\V^{\bar{\zeta}_m}$ has Hodge-Deligne diagrams

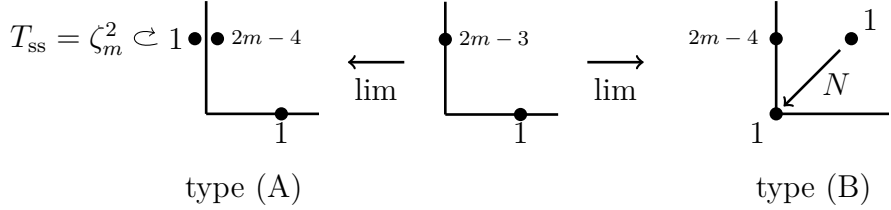
\begin{align*}
\begin{tikzpicture}[scale=0.5]
   	\node at (-0.8,2) (foo) {$1$};
    \path[->] (foo) edge  [loop left] node {$T_{\text{ss}}=\zeta_m^2$} ();
    \draw[-,line width=1.0pt] (0,0) -- (0,3);
    \draw[-,line width=1.0pt] (0,0) -- (3,0);
	\fill (2,0) circle (5pt);
	\node at (2,-0.5)  {$1$};		
	\fill (-0.3,2) circle (5pt);
	\fill (0.3,2) circle (5pt);	
	\node at (1.6,2) {\tiny$2m-4$};
	\node at (1,-2) {type (A)};	\end{tikzpicture}
& 
\begin{tikzpicture}[scale=0.5]
\draw[->,line width=1.0pt] (1.5,3.5) -- (0,3.5);
\node at (0.75,2.8) {$\lim$};
\node at (0,0) {\phantom{hh}};
\end{tikzpicture} 
&
\begin{tikzpicture}[scale=0.5]
    \draw[-,line width=1.0pt] (0,0) -- (0,3);
    \draw[-,line width=1.0pt] (0,0) -- (3,0);
	\fill (2,0) circle (5pt);
	\fill (0,2) circle (5pt);		
	\node at (2,-0.5) {$1$};
	\node at (1.3,2) {\tiny$2m-3$};
	\node at (1,-2) {\phantom{$\V^{\bar{\zeta}_m}$}};				
\end{tikzpicture}
& 
\begin{tikzpicture}[scale=0.5]
\draw[->,line width=1.0pt] (0,3.5) -- (1.5,3.5);
\node at (0.75,2.8) {$\lim$};
\node at (0,0) {\phantom{hh}};
\end{tikzpicture} 
&
\begin{tikzpicture}[scale=0.5]
   	\node at (1.6,0.8) (foo) {$N$};
    \draw[-,line width=1.0pt] (0,0) -- (0,3);
    \draw[-,line width=1.0pt] (0,0) -- (3,0);
    \draw[->,line width=1.0pt] (1.7, 1.7) -- (0.2,0.2);
	\fill (0,2) circle (5pt);
	\fill (0,0) circle (5pt);	
	\node at (2.5,2.5) {$1$};	
	\fill (2,2) circle (5pt);
	\node at (-1.4,2) {\tiny$2m-4$};		\node at (-0.5,-0.5) {$1$};
	\node at (1,-2) {type (B)};
\end{tikzpicture}
\end{align*}

\noindent For $m=4$ resp.~$6$, the monodromy in type (A) is thus given by a complex reflection resp.~``triflection''.
\end{example}

\begin{rem}\label{D7}
For any $m$, $\V^{\bar{\zeta}_m}\,(\oplus \V^{\zeta_m})$ induces a map from the universal cover $\widetilde{M}_{0,2m}^{\text{un}}$ to a ball $\BB_{2m-3}$. Moreover, both LMHS types have $2m-4$ complex moduli. However, for $m$ different from $2,3,4,$ or $6$, this does not lead to a tidy extended period map: as the projection of the monodromy to $U(1,2m-3)$ is not discrete \cite{Mo}, the quotient of $\BB_{2m-3}$ by this is not Hausdorff. 

To circumvent this problem, we must replace $\BB_{2m-3}$ by its product with other (non-ball) symmetric domains, which receives the image of the period map for the $\QQ$-VHS $\oplus_{(j,m)=1}\V^{\zeta_m^j}$.  For instance, if $m=5$ then the real points of the generic Mumford-Tate group of $\V$ take the form $U(1,7)\times U(3,5)$, and the full period map lands in a discrete quotient of the product $\BB_7\times \mathrm{I}_{3,5}$.
\end{rem}

%%%%%%%%%%%%%%%%%%%%%%%%%%%%%%%%%%%%%%
%%%%%%%%%%%%%%%%%%%%%%%%%%%%%%%%%%%%%%
\section{Hyperplane configurations and Dolgachev's conjecture}\label{S5}

Both differential and asymptotic methods in Hodge theory can be used to establish that a VHS is ``generic'' in some sense.  In \cite{GSVZ}, differential methods (characteristic varieties and Yukawa couplings) were employed to show that the period map for the family of CY 3-folds $X\overset{2:1}{\twoheadrightarrow}\PP^3$ branched over 8 planes does not factor through a locally symmetric variety of the form $\Gamma\backslash \mathrm{SU}(3,3)/K$.  Indeed, the geometric monodromy and Mumford-Tate groups of the corresponding VHS turn out to be as large as they can be (with both equal to the symplectic group $\mathrm{Sp}_{20}$).  This was later extended to similarly constructed CY $n$-fold families  \cite{SXZ}, see below.  Our goal here is to quickly deduce these results using eigenspectra and local monodromy, demonstrating the effectiveness of the asymptotic approach.

Let $L_0,\ldots,L_{2n+1}\subset \PP^n$ be hyperplanes defined by linear forms $\ell_i$, in general position in the sense that $\cup L_i$ is a normal crossing divisor. Consider the $2{:}1$ cover $X\overset{\pi}{\twoheadrightarrow}\PP^n$ branched along $\cup L_i$, and the rank-$1$ $\QQ$-local system $\LL$ on $U=\PP^n\sm(\cup L_i)\overset{\jmath}{\hookrightarrow}\PP^n$ with monodromy $-1$ about each $L_i$. Since $X$ has finite quotient singularities, we have $\mathrm{IC}^{\bullet}_X=\QQ_X[n]$ and\footnote{See \cite[Prop.~8.2.30]{HTT} for the statement that $\mathrm{IC}^{\bullet}_{\PP^n}\LL=\jmath_*\LL[n]$.}
\begin{equation}\label{E1}
H:=H^n_{\pr}(X):=\frac{H^n(X)}{\pi^*H^n(\PP^n)}\cong H^n(\PP^n,\jmath_*\LL)\cong \IH^n(\PP^n,\LL)
\end{equation}
is a pure HS of weight $n$. By \cite[Lemma 8.2]{DK}, it has Hodge numbers
\begin{equation}\label{E2}
h^{p,n-p}_{\pr}(X)=\binom{n}{p}^2\;\;\implies\;\;h^n_{\pr}(X)=\binom{2n}{n}.	
\end{equation}
It is polarized by the intersection form $Q$, which presents no difficulties as $X$ has a smooth finite cover.

Taking $\cs\subset (\check{\PP}^n)^{2n+2}/\mathrm{PGL}_{n+1}(\CC)=:\csb$ to be the ($n^2$-dimensional) moduli space of $2n+2$ ordered hyperplanes in $\PP^n$ in general position, this construction yields a $\ZZ$-PVHS $\H\to \cs$ of CY-$n$ type with $H$ as reference fiber.  Let $\rho\colon \pi_1(\cs)\to \mathrm{Aut}(H,Q)^{\circ}=:M_{\max}$ be the monodromy representation of $\H$,\footnote{Here $(\cdot)^{\circ}$ means the identity component as algebraic group (i.e.~$\mathrm{SO}(H)$ instead of $O(H)$ if $n$ is even).} $\Pi$ its geometric monodromy group, and $M$ its Hodge (special Mumford-Tate) group. Here $\Pi$ is the identity connected component of $\widetilde{\Pi}:=\overline{\rho(\pi_1(\cs))}^{\QQ\text{-Zar}}$, and $\Pi\leq M\leq M_{\max}$. A conjecture attributed by \cite{SXZ} to Dolgachev states that the period map for $\H$ factors through a locally symmetric variety (also $n^2$-dimensional) of type $I_{n,n}$,\footnote{Note that the ``tautological VHS'' over $I_{n,n}$ is already geometrically realized by the $n^{\text{th}}$ primitive cohomology of a universal family of Weil abelian $2n$-folds.} which would imply that $\mathfrak{m}_{\RR}\cong \mathfrak{su}(n,n)$. This is equivalent to saying that,
\begin{equation}\label{E3}
\begin{split}
\text{\emph{up to finite data} (i.e.~after passing to a finite cover),}\phantom{hhh}	\\ \text{\emph{$\H$ is the $n^{\text{th}}$ wedge power of a VHS of weight $1$ and rank $2n$.}}
\end{split}
\end{equation}

The conjecture does hold for $n=1$ and $n=2$, but this merely reflects exceptional isomorphisms of Lie groups in low rank, namely $\mathrm{SU}(1,1)\cong \mathrm{SL}_2(\RR)$ and $\mathrm{SU}(2,2)\cong\mathrm{Spin}(2,4)^+$. That is, in both of these cases we \emph{also} have $\Pi\cong M_{\max}$ ($=\mathrm{SL}_2$ resp.~$\mathrm{SO}(2,4)$). For $n\geq 3$, in contrast, the conjecture would have $\Pi<M_{\max}$ a proper algebraic subgroup. In [op.~cit.] (and earlier works \cite{GSVZ,GSZ,GSZ2}), it was shown via quite computationally involved differential methods that in fact the monodromy is maximal for all $n$, and the conjecture fails for $n\geq 3$:

\begin{thm}\label{E4}
$\Pi=M=M_{\max}$ $\forall n\geq 1$.
\end{thm}

In the remainder of this section, we explain how asymptotic methods provide a much simpler approach to these results.  First we will give a careful argument disproving the conjecture for $n\geq 3$ odd, which \emph{a priori} is a weaker statement than the Theorem in that case. (The relation to the main theme of his paper --- specifically, to the setting of Cor.~\ref{A9} --- enters when we pass to the smooth finite cover $\hat{X}$ of $X$.) Then we sketch a proof of Theorem \ref{E4} using a more topological and monodromy-theoretic approach.

\subsection*{Disproof of \eqref{E3} for $n$ odd}

Most of the analysis that follows works for all $n$, though the last step is inconclusive for even $n$.

To begin, consider a pencil $\PP^1\overset{\ve}{\hookrightarrow}\csb$ of hyperplane configurations given by fixing $L_0,\ldots,L_{2n}$ (in general position) and letting $L_{2n+1}:=L_s$ vary along a line in $\check{\PP}^n$ (chosen to avoid linear spans of any $n-2$ $L_i$ in $\check{\PP}^n$).\footnote{It already follows from Zariski's theorem \cite[Thm.~3.22]{Vo} that $\rho(\pi_1(\PP^1\sm\Sigma))=\rho(\pi_1(\cs))$, but we won't need this.} Writing $\Sigma=\ve^{-1}(\csb\sm\cs)$, we have $|\Sigma|=\binom{2n+1}{n}$; and degenerations $\X_{\sigma}\to \Delta_{\sigma}$ of our double-covers at $\sigma\in \Sigma$ are locally modeled (with $t=s-\sigma$) by 
\begin{equation}\label{E5}
w^2\underset{\text{loc}}{=}x_1\cdots x_n(t-x_1-\cdots -x_n)	
\end{equation}
after a $\mathrm{PGL}_{n+1}(\CC)$-action. Accordingly, writing $X_0,\ldots,X_n$ for projective coordinates on $\PP^n$, we take $\ell_i=X_i$ for $0\leq i\leq n$ and $\ell_{2n+1}=tX_0-\sum_{i=1}^n X_i$, and $\ell_{n+1},\ldots,\ell_{2n}$ ``general''.

Let $\ul\colon \PP^n\hookrightarrow \PP^{2n+1}$ denote the linear embedding $[X_0{:}\cdots{:}X_n]\mapsto [\ell_0(\UX){:}\cdots {:}\ell_{2n+1}(\UX)]$, and $\phi\colon \PP^{2n+1}\to \PP^{2n+1}$ denote the map sending $[Z_0{:}\cdots {:}Z_{2n+1}]\mapsto [Z_0^2{:}\cdots {:}Z_{2n+1}^2]$. Then $\hat{X}:=\phi^{-1}(\ul(\PP^n))\subset \PP^{2n+1}$ is a smooth complete intersection on which\footnote{Here $\NP$ denotes the diagonal embedding.} $\A:=(\ZZ/2\ZZ)^{2n+2}/\NP(\ZZ/2\ZZ)$ acts via $\ue^{(i)}\mapsto \{Z_i\mapsto -Z_i\}$, with quotient $\PP^n$; explicitly, we have
\begin{equation}\label{E6}
\hat{X}=\textstyle\bigcap_{k=0}^n\{0=F_k(\UZ):=-Z_{n+k+1}^2+\ell_{n+k+1}(Z_0^2,\ldots,Z_n^2)\}.
\end{equation}
Write $\chi\in \mathrm{X}^*(\A)$ for the character sending each $\ue^{(i)}\mapsto -1$, $\A^{\circ}:=\ker(\chi)\leq \A$, and $q\colon \hat{X}\twoheadrightarrow X$ for the quotient by $\A^{\circ}$; then $H\cong q^* H^n_{\pr}(X)\cong H^n(\hat{X})^{\chi}$. Since $F_0(\UZ)=tZ_0^2-\sum_{i=1}^{n+1}Z_i^2$, we have thus replaced our original \emph{non}-isolated degeneration \eqref{E5} by a nodal one.

Next, we use the ``Cayley trick'' to replace the complete intersection $\hat{X}$ by a hypersurface
\begin{equation}\label{E7}
Y:=\{0=F:=\textstyle\sum_{k=0}^n Y_k F_k(\UZ)\}\subset \PP\left(\co_{\PP^{2n+1}}(2)^{\oplus n+1}\right)=:\mathbf{P}
\end{equation}
of dimension $3n$.  We have an $\A$-equivariant isomorphism $H^{3n}(Y)(n)\cong H^n(\hat{X})$ of HSs, so that $H\cong H^{3n}(Y)^{\chi}(n)$.  Notice that in affine coordinates $(z_1,\ldots,z_{2n+1};y_1,\ldots,y_n)$, $F=0$ becomes\footnote{Here ``h.o.t.'' means terms vanishing to order $3$ at the nodes.}
\begin{equation}\label{E8}
0=t-z_1^2-\cdots - z_{n+1}^2+\textstyle \sum_{k=1}^n y_k(b_k-z_{n+k+1})(b_k+z_{n+k+1})+\text{h.o.t},	
\end{equation}
where $b_k:=\sqrt{F_k(1,0,\ldots,0)}$. So at $t=0$, the singular fiber $Y_{\sigma}$ has $2^n$ nodes at $(Z_0;Z_1,\ldots,Z_{n+1};Z_{n+2},\ldots,Z_{2n+1};Y_0;Y_1,\ldots,Y_n)=$
\begin{equation}\label{E9}
\left(1;0,\ldots,0;(-1)^{a_1} b_1,\ldots,(-1)^{a_n}b_n;1;0,\ldots,0\right),\;\;\;\ua\in (\ZZ/2\ZZ)^n,	
\end{equation}
and the degeneration $\Y_{\sigma}\to \Delta_{\sigma}$ has smooth total space. The mixed spectrum of each node is $[(\tfrac{3n+1}{2},3n+1)]$ for $n$ odd and $[(\tfrac{3n+1}{2},3n)]$ for $n$ even; so $T_{\sigma}$ acts through multiplication by $(-1)^{n+1}$ on
\begin{equation}\label{E10}
H^{3n}_{\van}(Y_t)\cong \QQ(-\lfloor \tfrac{3n+1}{2} \rfloor)^{\oplus 2^n}.	
\end{equation}
Moreover, since the summands of \eqref{E10} are represented by $\eta_{\ua}=(-1)^{|\ua|}(dz_1\wedge \cdots \wedge dz_{2n+1}\wedge dy_1\wedge\cdots \wedge dy_n)/F^{\lceil\frac{3n+1}{2}\rceil}$ near the nodes \eqref{E9} (in the sense of \cite[\S2]{KL2}), it has a 1-dimensional subspace (generated by $\eta_{\chi}:=\sum (-1)^{|\ua|}\eta_{\ua}$) on which $\A$ acts through $\chi$.

Taking $\chi$-eigenspaces of the vanishing-cycle sequence for $\Y_{\sigma}\to \Delta_{\sigma}$ and twisting by $\QQ(n)$ now yields\small
\begin{equation}\label{E11}
0\to H^{3n}(Y_{\sigma})^{\chi}(n)\overset{\sp^{\chi}}{\to} \underset{\cong H_{\lim}}{\underbrace{H^{3n}_{\lim}(Y_t)^{\chi}(n)}}\overset{\can^{\chi}}{\to}\QQ(-\lfloor\tfrac{n+1}{2}\rfloor)\overset{\delta^{\chi}}{\to} H^{3n+1}_{\text{ph}}(Y_{\sigma})^{\chi}(n)\to 0	
\end{equation}\normalsize
We claim that $\delta=0$.  For $n$ even, this is clear, since $T_{\sigma}$ acts trivially on $H^{3n+1}_{\text{ph}}(Y_{\sigma})$ and by $-1$ on $\QQ(-\lfloor \tfrac{n+1}{2}\rfloor)$. So we conclude that $T_{\sigma}$ acts on $H_{\lim}$ via an orthogonal reflection. This doesn't factor through $\bigwedge^n$ of any automorphism of $\CC^{2n}$, but because it is finite (of order $2$), this does not (yet) disprove the conjecture.

On the other hand, for $n$ odd, it is not automatic that $\delta=0$. (This is a well-known problem with nodal degerenations in odd dimensions, cf.~\cite[\S2.2]{KL2}; and as we saw in the proof of Lemma \ref{E5}, our degenerations are finite quotients of nodal ones.) But if we can show $\delta=0$, then the conjecture is \emph{immediately} disproved (for odd $n\geq 3$).  Here is why: by \eqref{E6}, $H_{\lim}$ then has a class of type $(n+1,n+1)$, which must go to an $(n,n)$ class by $N_{\sigma}$,
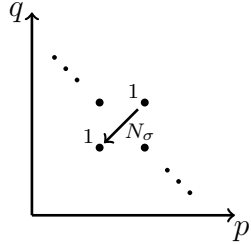
\begin{center}
\begin{equation*}
	\begin{tikzpicture}[scale=0.3]
   	\node at (3.8,2.7) (foo) {\tiny$N_{\sigma}$};
    \draw[->,line width=1.0pt] (-1,-1) -- (-1,8);
    \draw[->,line width=1.0pt] (-1,-1) -- (8,-1);
    \draw[->,line width=1.0pt] (3.7, 3.7) -- (2.2,2.2);
	\fill (4,2) circle (5pt);
	\fill (2,4) circle (5pt);	
	\node at (1.5,2.5) {\tiny$1$};
	\fill (2,2) circle (5pt);
	\fill (4,4) circle (5pt);
	\node at (3.5,4.5) {\tiny$1$};
	\node at (8.3,-1.7) {$p$};
	\node at (-1.7,8) {$q$};
	\fill (1,5) circle (3pt);
	\fill (0.5,5.5) circle (3pt);
	\fill (0,6) circle (3pt);
	\fill (5,1) circle (3pt);
	\fill (5.5,0.5) circle (3pt);
	\fill (6,0) circle (3pt);
\end{tikzpicture}
\end{equation*}
\end{center}
forcing $\mathrm{rk}(N_{\sigma})=1$ (rather than $0$). (In different terms, each $T_{\sigma}$ is a nontrivial symplectic transvection.) But this is impossible for $\bigwedge^n$ of a nilpotent endomorphism of $\CC^{2n}$.

To complete the (dis)proof, then, we apply \cite[Thm.~2.9]{KL2}: for a nodal degeneration $Y\rightsquigarrow Y_{\sigma}$ of an odd-dimensional hypersurface of a smooth projective variety $\mathbf{P}$ satisfying Bott vanishing, the rank of $\delta$ is the number $m$ of nodes minus the rank of the map $$\mathrm{ev}\colon H^0(\mathbf{P},K_{\mathbf{P}}(\tfrac{3n+1}{2}Y_{\sigma}))\to \CC^m$$ given by evaluation at the nodes. The proof in [loc.~cit.] is equivariant in $\A$, and so we find that $\delta^{\chi}=0$ $\iff$ $\mathrm{ev}$ is nonzero on $H^0(\mathbf{P},K_{\mathbf{P}}(\tfrac{3n+1}{2}Y_{\sigma}))^{\chi}$, which can be checked at any node. Writing $\vec{\mathsf{e}}_1:=\sum_{i=0}^nY_i\tfrac{\d}{\d Y_i}$, $\vec{\mathsf{e}}_2:=\sum_{j=0}^{2n+1}Z_j\tfrac{\d}{\d Z_j}-2\vec{\mathsf{e}}_1$, and $\Omega:=\langle \vec{\mathsf{e}}_2,\langle \vec{\mathsf{e}}_1,d\UZ\wedge d\underline{Y}\rangle\rangle$, one checks that 
\begin{equation}\label{E12}
Y_0Z_0^2\Omega/(F_{t=0})^{\frac{3n+1}{2}}	
\end{equation}
is a well-defined section of $K_{\mathbf{P}}(\tfrac{3n+1}{2}Y_{\sigma})$ (cf.~\cite[\S4.5]{Ke}); and evidently $\A$ acts on it through $\chi$. Clearly, it is nonzero on the fiber of $K_{\mathbf{P}}(\tfrac{3n+1}{2}Y_{\sigma})$ at any of the nodes \eqref{E9}.

\subsection*{Sketch of proof of Theorem \ref{E4}}

Returning to the local picture \eqref{E5}, we now seek a more concrete topological description of the orthogonal reflections ($n$ even) and symplectic transvections ($n$ odd) through which $T_{\sigma}$ acts on $H$. So let $U_0\subset \mathbb{A}^n$ be the complement of the hyperplanes $x_1=0,\ldots,x_n=0$ and $x_1+\cdots +x_n=1$, and $\LL_0$ the rank-$1$ local system on $U_0$ with monodromies $-1$ about each of them. While the singularity $x_{\sigma}\overset{\imath_{\sigma}}{\hookrightarrow}X_{\sigma}$ ``at $\uo$'' in \eqref{E5} isn't isolated, the vanishing-cycle complex $\phi_t\QQ_{\X}$ is nothing but $\imath^{\sigma}_*V[-n]$, where $V:=\IH^n(\mathbb{A}^n,\LL_0)$ (as MHS). We begin with a local analogue of the covering argument just seen.

\begin{lem}\label{E13}
\textup{(i)} $\IH^n(\mathbb{A}^n,\LL_0)\cong \QQ(-\lfloor\tfrac{n+1}{2}\rfloor).$

\textup{(ii)} Local monodromy $T_{\sigma}$ acts on $V$ through multiplication by $(-1)^{n+1}$.

\textup{(iii)} The canonical map $\can_{\sigma}\colon H_{\lim}\to V$ is onto.
\end{lem}
\begin{proof}
Define maps 
\begin{itemize}[leftmargin=0.5cm]
\item $f_0\colon \mathbb{A}^n\hookrightarrow \mathbb{A}^{n+1}$ by $\ux\mapsto (\ux,1-\sum_{i=1}^n x_i)$, and
\item $\phi_0\colon \mathbb{A}^{n+1}\to\mathbb{A}^{n+1}$ by squaring all coordinates $z_i$;
\end{itemize}
then $\hat{X}_0:=\phi_0^{-1}(f_0(\mathbb{A}^n))\subset \mathbb{A}^{n+1}$ is the quadric hypersurface $\sum_{i=1}^{n+1}z_i^2=1$. The group $\A_0:=(\ZZ/2\ZZ)^{n+1}/\NP(\ZZ/2\ZZ)$ acts on $\hat{X}_0$ (multiplying coordinates by $\pm1$), with quotient $\mathbb{A}^n$. The quotient $q_0\colon \hat{X}_0\twoheadrightarrow X_0$ by the augmentation subgroup $\A_0^{\circ}$ yields the obvious $2{:}1$ branched cover of $\mathbb{A}^n$, with $H^n(X_0)\cong \IH^n(\mathbb{A}^n,\LL_0)$.

By the localization sequence for $\hat{X}_0$ (relative to its closure $\overline{\hat{X}}_0\subset \PP^{n+1}$) and weak Lefschetz, one easily shows that $H^j(\hat{X}_0)=0$ for $j\neq n$,\footnote{This simply recovers perversity of $\phi_f\QQ_{\X}[n]$.} and $H^n(\hat{X}_0)\cong \QQ(-\lfloor\tfrac{n+1}{2}\rfloor)$. (Writing $\d\hat{X}_0=\overline{\hat{X}}_0\sm\hat{X}_0$, this is $H^n(\overline{\hat{X}}_0)/H^{n-2}(\d\hat{X}_0)(-1)$ for $n$ even, and $\ker\{H^{n-1}(\d\hat{X}_0)(-1)\to H^{n+1}(\overline{\hat{X}}_0)\}$ for $n$ odd.) A generator for the \emph{dual} group $H^n_c(\hat{X}_0)$ is given by the real (vanishing) $n$-sphere $S^n_1:=\{\sum z_i^2=1\}\cap \RR^{n+1}$, whose class is invariant under $\A_0^{\circ}$ hence comes from $H^n_c(X_0)$. This gives (i).

The degeneration is modeled by replacing $\sum z_i^2=1$ by $\sum z_i^2=t$; as the spectrum of $\sum z_i^2$ is $[\tfrac{n+1}{2}]$, the monodromy is as described in (ii). Finally, (iii) follows from the last subsection since $\can_{\sigma}$ identifies with $\can^{\chi}$ in \eqref{E11}.
\end{proof}

The vanishing sphere $S^n_t:=\{\sum z_i^2=t\}\cap \RR^{n+1}$ in $\hat{X}_0$ has image in $X_0$ (by $q_0$) given by the double cover of $(\cap_{i=1}^n\{x_i\geq 0\})\cap \{\sum x_i\leq t\}$.  Let its image in $X$ (essentially via $\can^{\chi}\colon H^n_c(X_0)\to H^n(X)$) be denoted by $\nu_{\sigma}$; this is the vanishing cycle at $\sigma$, a ``double simplex'' branched along $\H_s$ and $n$ additional hyperplanes. It follows from (iii) that $T_{\sigma}$ is a transvection/reflection in $\nu_{\sigma}$. More precisely, rescaling $Q$ to have $Q(\nu_{\sigma},\nu_{\sigma})=\tfrac{1+(-1)^n}{2}$,
\begin{equation}\label{E14}
T_{\sigma}(u)=u-2Q(u,\nu_{\sigma})\nu_{\sigma}	
\end{equation}
for $u\in H$.

Now consider the general setting where $L_{2n+1}=L_s$, $L_0=\{X_0=0\}$, and the remaining $L_i$ are in general position. An easy extension of \eqref{E1} gives $H\cong \IH^n_c(\mathbb{A}^n,\LL)\cong H^n_c(X\sm L_0)$, whence $H^n_{\mathrm{pr}}(X)$ is spanned by double simplices branched along $n+1$ of the $L_{i\geq 0}$. Obviously all of these can be rewritten as $\ZZ$-linear combinations of double simplices branched along $L_s$ and $n$ of the $\{L_i\}_{1\leq i\leq 2n}$; call these $\nu_I$, where $I\subset \{1,\ldots,2n\}$ with $|I|=n$. Since $\mathrm{rk}H=\binom{2n}{n}$ and there are $\binom{2n}{n}$ of these vanishing cycles, they form a $\QQ$-basis of $H=H_{\text{pr}}^n(X)$. Write $T_I$ for the corresponding monodromies, and $\Gamma\leq \mathrm{Aut}(H_{\CC},Q)$ for the smallest $\CC$-algebraic group containing them; clearly $\Gamma\leq \widetilde{\Pi}_{\CC}$. Moreover, we note that if $|I\cap I'|=n-1$, then $Q(\nu_I,\nu_{I'})=\pm 1$ (rescaling as above, compatibly with \eqref{E14}).

Suppose then that $|I\cap I'|=n-1$.  If $n$ is odd, then $T_I(\nu_{I'})=\nu_I\pm \nu_{I'}=\pm T_{I'}^{-1}(\nu_I)$, whence $\nu_{I'}$ is in the $\Gamma$-orbit of $\nu_I$; so \emph{all} the $\nu_J$ are in the $\Gamma$-orbit of $\nu_I$. If $n$ is even, then reasoning as in \cite[\S4.4]{De},\footnote{See the paragraph after Lemme $4.4.3^{\text{s}}$.} $T_I T_{I'}^{\pm 1}$ is a transvection and its Zariski closure a $\mathbb{G}_a$ including transformations which send $\nu_I \mapsto \nu_{I'}$ and vice-versa; once again, all the $\nu_J$ are in the $\Gamma$-orbit of a single $\nu_I$.

Let $R:=\Gamma.\nu_I$ denote this orbit.  Obviously it spans $H_{\CC}$. Furthermore, for any $\delta\in R$, $\Gamma$ contains the transvection/reflection $T_{\delta}$: writing $\delta=\gamma.\nu_I$ ($\gamma\in \Gamma$), we have $T_{\delta}=T_{\gamma.\nu_I}=\gamma T_I\gamma^{-1}\in \Gamma$. So $\Gamma$ is in fact the $\CC$-algebraic closure of the $\{T_{\delta}\}_{\delta \in R}$, and we are exactly in the situation of \cite[Lemme 4.4.2]{De}. Conclude that $\Gamma=\mathrm{Aut}(H_{\CC},Q)$, hence $\widetilde{\Pi}=\mathrm{Aut}(H,Q)$, and thus $\Pi=\mathrm{Aut}(H,Q)^{\circ}$, proving Theorem \ref{E4}.

\begin{rem}\label{E15}
After writing this paper we encountered the article by \cite{Xu} which treats the more general setting of $r$-covers of $\PP^n$ branched along hyperplanes by considering local monodromies (as we have just done).  The argument is necessarily more complicated and technical than ours.  However, in the case $r=2$ (i.e., our setting) it appears to be incomplete.

If $r=2$ and $n$ is odd,  \cite[Prop.~3.4]{Xu} does not actually establish that $e_{(1)}$ (in the notation of [loc.~cit.]) is nonzero; this is exactly the issue regarding possible nonvanishing of $\delta$ dealt with above.  One could read \cite[Prop.~4.2]{Xu} as confirming this in retrospect, but this makes the argument quite convoluted.

If $r=2$ and $n$ is even, the proof of \cite[Prop.~4.2]{Xu} is wrong, as it makes use of the (false) statement that $\mathrm{Sp}_{2n}(\RR)$ ``does not admit any nontrivial one-dimensional invariant subspace'' in its action on $\bigwedge^n\RR^{2n}$.
\end{rem}

\end{document}